\documentclass[preprint]{imsart_beta}

\RequirePackage[OT1]{fontenc}
\usepackage{amsthm,amsmath}
\RequirePackage[numbers]{natbib}
\RequirePackage[colorlinks,citecolor=blue,urlcolor=blue]{hyperref}

\pubyear{0000}
\volume{0}
\firstpage{0}
\lastpage{0}

\startlocaldefs

\numberwithin{equation}{section}
\theoremstyle{plain}
\newtheorem{thm}{Theorem}[section]
\theoremstyle{remark}
\newtheorem{rem}{Remark}[section]

\theoremstyle{plain}

\newtheorem{proposition}{Proposition}[section]

\theoremstyle{definition}
\newtheorem{definition}{Definition}[section]

\endlocaldefs

\usepackage{dblfloatfix}    
\usepackage{kantlipsum}     

\usepackage{graphicx, enumitem}

\usepackage[utf8]{inputenc}
\usepackage[english]{babel}
\usepackage{amssymb}
\usepackage{amsfonts}
\usepackage{bm}
\usepackage{graphicx}
\usepackage{xcolor}
\usepackage{parskip}
\usepackage{multicol}
\usepackage{float}
\usepackage{rotating}

\def\Re{{\rm Re}}
\def\Im{{\rm Im}}

\begin{document}

\begin{frontmatter}

\title{Regularity Properties and Simulations of Gaussian Random Fields on the Sphere cross Time 
}
\runtitle{Regularity and simulations of GRFs on Spheres cross Time.}


\author{\fnms{Jorge} \snm{Clarke De la Cerda} \thanksref{t2}  \corref{} \ead[label=e1]{clarkemove@gmail.com}}
\thankstext{t2}{supported by {\it Proyecto FONDECYT Post-Doctorado Nº 3150506}.}
\address{Departamento de Matemática, Universidad Técnica Federico Santa María.\\ Avda. Espa\~na 1680, Valparaíso, Chile.\\ 
CEREMADE, UMR CNRS 7534 Université Paris-Dauphine, PSL Research university, Place du Maréchal
de Lattre de Tassigny 75016 Paris, FRANCE.  \\
 \printead{e1}}

\author{\fnms{Alfredo} \snm{Alegría} \thanksref{t3}\ead[label=e2]{alfredo.alegria.jimenez@gmail.com}}
\thankstext{t3}{supported by {\it Beca CONICYT-PCHA / Doctorado Nacional / 2016-21160371}.}
\address{Newcastle University, School of Mathematics and Statistics.\\ 
 \printead{e2}}

\and
\author{\fnms{Emilio} \snm{Porcu}\thanksref{t4}
\ead[label=e3]{georgepolya01@gmail.com}
\ead[label=e4,url]{eporcu.mat.utfsm.cl}}
\thankstext{t4}{supported by {\it Proyecto FONDECYT Regular Nº 1170290}.}
\address{ Newcastle University, School of Mathematics and Statistics.\\ Chair of Spatial Analytics Methods Center.\\
Departamento de Matemática, Universidad Técnica Federico Santa María. \\
Avda. España 1680, Valparaíso, Chile. \\
\printead{e3} \\ \printead{e4}}


\runauthor{J. Clarke et al.}

\begin{abstract}
    We study the regularity properties of Gaussian fields defined over spheres cross time. In particular, we consider two alternative spectral decompositions for a Gaussian field on $\mathbb{S}^d \times \mathbb{R}$. For each decomposition, we establish regularity properties through Sobolev and interpolation spaces. We then propose a simulation method and study its level of accuracy in the $L^2$ sense. The method turns to be both fast and efficient. 
\end{abstract}





\end{frontmatter}

\section{Introduction}

	Spatio-temporal variability is of major importance in many fields, in particular for anthropogenic and natural processes, such as earthquakes, geographic evolution of diseases, income distributions, mortality fields, atmospheric pollutant concentrations, hydrological basin characterization and precipitation fields, among others.  
	For many natural phenomena involving, for instance, climate change and atmospheric variables, many branches of applied sciences have been increasingly interested in the analysis of data distributed over the whole sphere representing planet Earth and evolving through time. Hence the need for random fields models where the spatial location is continuously indexed through the sphere, and where time can be either continuous or discrete. It is common to consider the observations as a partial realization of a spatio-temporal random field which is usually considered to be Gaussian (see \cite{Christakos-2005, CH-1998, Dimitrakopoulos-1994}).   Thus, the dependence structure in space-time is governed by the covariance of the spatio-temporal Gaussian field,  and we refer the reader to \cite{Gneiting:2002, Stein:2005, PZ:2011} for significant contributions in this direction.  

Specifically, let $d$ be a positive integer, and let $\mathbb{S}^{d} = \{\mathbf{x}\in\mathbb{R}^{d+1},\|\mathbf{x}\|=1\}$ be the $d$-dimensional unit sphere in the Euclidean space $\mathbb{R}^{d+1}$, where $\|\cdot\|$ denotes the Euclidean norm of $\mathbf{x} \in \mathbb{R}^{d+1}$. We denote $Z = \{Z(\mathbf{x},t),\ (\mathbf{x},t)\in \mathbb{S}^{d}\times \mathbb{R}\}$ a Gaussian field on $\mathbb{S}^d \times \mathbb{R}$.
   

The {\em tour de force} in \cite{LS-2015} characterizes isotropic Gaussian random fields on the sphere $\mathbb{S}^{d}$ through Karhunen–Loève expansions with respect to the spherical harmonics
functions and the angular power spectrum. They show that the smoothness
of the covariance is connected to the decay of the angular power
spectrum and then discuss the relation to sample H{\"o}lder continuity and sample
differentiability of the random fields.

The present paper extends part of the work of \cite{LS-2015} to space-time. Such extension is non-trivial and depends on two alternative spectral decompositions of a Gaussian field on spheres cross time. In particular, we propose either Hermite or classical Karhunen-Loève expansions, and show how regularity properties can evolve dynamically over time. The crux of our arguments rely on recent advances on the characterization of covariance functions associated to Gaussian fields on spheres cross time (see \cite{Berg2016} and \cite{BGP-2015}). Notably, the Berg-Porcu representation in terms of Schoenberg functions inspires the proposal of alternative spectral decompositions for the temporal part, which become then crucial to establish the regularity properties of the associated Gaussian field.

The second part of the paper is devoted to simulation methods which should be computationally fast while keeping a reasonable level of accuracy, resulting in a notable step forward. Efficient simulation methods for random fields defined on the sphere cross time are, currently,  almost unexplored.   Cholesky decomposition is an appealing alternative since it is an exact method. However, the method has an order of computation of $\mathcal{O}(N^3)$, with $N$ denoting the sample size. This makes its implementation computationally challenging for large scale problems (the so called Big \textquotedblleft{\it n}\textquotedblright problem).   It is therefore mandatory to investigate efficient simulation methods. Here, we propose a simulation method based on a suitable truncation of the proposed double spectral decompositions. We establish its accuracy in the $L^2$ sense and illustrate how the model keeps a reasonable level of accuracy while being considerably fast, even when the number of spatio-temporal locations is very high.

The remainder of the article is as follows.   Section \ref{Preli} provides the basic material for a self-contained exposition.   The expansions for the kernel covariances and the random field are presented in Section \ref{Expansions}.   Section \ref{Reg_Prop} presents the regularity results for the kernel covariance functions in terms of weighted Sobolev spaces and weighted bi-sequence spaces. In Section \ref{Simul} our simulation method is developed and its accuracy is studied. Also, we provide some numerical experiments for illustrative purposes. In Appendix \ref{Apendix_KL} we also provide a rather general version of the Karhunen-Loève theorem.


The manuscript is intended for complex-valued random fields over $\mathbb{S}^{d} \times \mathbb{R}$ with  $d \in \mathbb{N}$, except for Section \ref{Simul} where the simulations considered are for real-valued random fields over $\mathbb{S}^{2} \times \mathbb{R}$.


\section{Preliminaries}\label{Preli}

	This section is largely expository and devoted to the illustration of the framework and notations that will be of major use throughout the manuscript. 
All the tools presented in this section are valid in $\mathbb{S}^{d}$, for any $d \in \mathbb{N}$. Some particular references to the case $d = 2$ are exposed, as they will be of use in Section \ref{Simul}.


\subsection{Spherical Harmonics Functions and Gegenbauer Polynomials} \label{sec_sph_harm}

	Spherical harmonics are restrictions to the unit sphere $\mathbb{S}^{d}$ of real harmonics polynomials in $\mathbb{R}^{d+1}$. They are also the eigenfunctions of the Laplace-Beltrami operator on $\mathbb{S}^{d}$. A deeper overview on spherical harmonics along with the properties listed in this subsection can be found in \citep{dai2013approximation}. 
	
	For $d \in \mathbb{N}$, let $L^{2}(\mathbb{S}^{d},\sigma_{d}; \mathbb{C}) = L^{2}(\mathbb{S}^{d} )$ denote the space of complex-valued square integrable functions over $\mathbb{S}^{d}$, where $\sigma_{d}$ denotes the surface area measure, and $\| \sigma_{d} \|$ denotes the surface area of $\mathbb{S}^{d}$, 
\begin{equation*}
\displaystyle
	\| \sigma_{d} \| := \int_{\mathbb{S}^{d}} \ {\rm d}\sigma_{d} = \frac{2 \pi^{(d+1)/2}}{\Gamma((d+1)/2)},
\end{equation*}
with $\Gamma$ being the  Gamma function.

For $j = 0, 1, \ldots, $ let $\mathcal{H}_{ j }^{d}$ denote the linear space  of spherical harmonics of degree $j$ over $\mathbb{S}^{d}$. For different degrees, spherical harmonics are orthogonal with respect to the inner product of $L^{2}(\mathbb{S}^{d})$. If $Y_{j, m, d} \in \mathcal{H}_{ j }^{d}$ and $Y_{j ', m, d} \in \mathcal{H}_{ j' }^{d}$, then
	
\begin{equation*}
\begin{split}
\displaystyle
	\left\langle Y_{j, m, d} , Y_{j ', m, d} \right\rangle_{L^{2}(\mathbb{S}^{d})} & := \int_{\mathbb{S}^{d}} \ Y_{j, m, d}(\mathbf{x}) \overline{ Y_{j ', m, d}(\mathbf{x}) } \  {\rm d}\sigma_{d} (\mathbf{x})
	\\
	& = \delta_{j, j '} \cdot \| \sigma_{d} \|;  \ \ \  j, j' \in \mathbb{N}, \  m = 0, \ldots , \dim (\mathcal{H}_{ j }^{d}),
\end{split}
\end{equation*}	
where $\delta_{j,j'}$ is the Kronecker delta function, being identically equal to one if $j=j'$, and zero otherwise.	

Corollary 1.1.4 in \cite{dai2013approximation} shows that 
\begin{equation}\label{SH_space_dim}
	\dim (\mathcal{H}_{ j }^{d}) = \frac{(2 j +d -1) (j +d -2)!}{j ! (d-1)!}, \ \ \ j \geq 1, \ \ \  \dim \mathcal{H}_{0}^{d} = 1.
\end{equation}	
	

Let $ \left\lbrace \mathcal{Y}_{j, m, d} : m = 1, \hdots, \dim (\mathcal{H}_{ j }^{d}) \right\rbrace $ be any orthonormal basis of $\mathcal{H}_{ j }^{d}$. Then, the family $\mathcal{S}_{d} := \left\lbrace \mathcal{Y}_{j, m, d} : j \in \mathbb{N}_{0}, m = 1, \hdots, \dim (\mathcal{H}_{ j }^{d}) \right\rbrace $ constitutes an orthonormal basis of $L^{2}(\mathbb{S}^{d})$ and Theorem 2.42 of \cite{Morimoto-1998} shows that 
\begin{equation*}
\displaystyle
		L^{2}(\mathbb{S}^{d}) = \bigoplus _{j = 0}^{+\infty} \mathcal{H}_{ j }^{d}.
\end{equation*}
 Besides, the addition formula for spherical harmonics states 
\begin{equation}\label{addition}
\sum_{m = 1}^{\dim (\mathcal{H}_{ j }^{d})} \mathcal{Y}_{j,m,d}(\mathbf{x}) \overline{ \mathcal{Y}_{j,m,d}(\mathbf{y}) } = \frac{(2j+d-1)}{(d-1)}  C_{j}^{(d-1)/2}( \langle \mathbf{x},\mathbf{y}\rangle_{\mathbb{R}^{d+1}}),  \ \ \  j \in \mathbb{N},
\end{equation}
where $\langle \cdot , \cdot \rangle_{\mathbb{R}^{d+1}}$ denotes the inner product in $\mathbb{R}^{d+1}$. Here $C_{j}^{r} (\cdot) $ are the Gegenbauer (or ultraspherical) polynomials of degree $j$ and order $r$, defined by
\begin{equation*}\label{Gegenbauer_Pol} 
	C_{j}^{r}(x) = \frac{(2 r)_{j}}{( r + 1/2)_{j}} P_{j}^{ ( r - 1/2 ) , ( r - 1/2 ) } (x), \qquad r > -1/2, \ x \in [-1, 1],
\end{equation*} 
where $P_{j}^{ ( \alpha, \beta ) } (\cdot) $ denotes the Jacobi polynomial of parameters $\alpha, \beta > -1$ and degree $j$, and $(\cdot)_j$ denotes the  Pochhammer symbol (rising factorial) defined by
\begin{equation*}\label{Pochhamer_symbol}
	(r)_{j} := r (r + 1) \cdot \ldots \cdot (r + j -1), \qquad r \in \mathbb{R}, \ j \in \mathbb{N}_{0}, 
\end{equation*}
where $(r)_{j} = \Gamma (r + j) / \Gamma (r)$, provided $r$ is not a negative integer. 

Gegenbauer polynomials constitute a basis of $L^{2}(-1,1)$ and satisfy the orthogonality relation (section 3.15 in \cite{BE-1953}) 
\begin{equation}\label{Gegen-ortho}
\displaystyle
	\int_{-1}^{1} C_{j}^{r}(x) C_{j'}^{r}(x) (1-x^{2})^{r - 1/2} {\rm d}x  =  \frac{\pi 2^{1-2r} \Gamma(j + 2r) }{j! (j+r) \Gamma(r)^{2}} \delta_{j, j'}. \ \ 
\end{equation} 
By Stirling's inequalities, for $n \in \mathbb{N}$ fixed, there exists constants $c_{1}(n)$ and $c_{2}(n)$ such that
\begin{equation*}
	c_{1}(n) j^{n-2} \leq \frac{\pi 2^{1-n} \Gamma(j + n) }{j! (j+n/2) \Gamma(n/2)^{2}} \leq c_{2}(n) j^{n-2}.
\end{equation*}
Hence, assuming $2r \in \mathbb{N}$, relation (\ref{Gegen-ortho}) becomes	
\begin{equation}\label{Gegen-ortho-2}
\displaystyle
	\int_{-1}^{1} C_{j}^{r}(x) C_{j'}^{r}(x) (1-x^{2})^{r - 1/2} dx  \simeq j^{2r - 2} \delta_{j,j'}.
\end{equation} 
Besides, for $n \leq j$ we observe that (see section 10.9 in \cite{BE-1953}) 
\begin{equation}\label{Gegen-deriv}
\displaystyle
	\frac{\partial^{n}}{\partial x^{n}} C_{j}^{r} (x) = 2^{n} (r)_{n} C_{j-n}^{r + n} (x) \simeq C_{j - n}^{r + n} (x), \ x \in [-1,1]
\end{equation}

In what follows $c_{j}(d, x)$ denotes the standardized Gegenbauer polynomial, being identically equal to one for $x = 1$ and $r = (d-1)/2$, that is
\begin{equation} \label{Gegen-norm}
	c_{j}(d, x) = \frac{C_{j}^{(d-1)/2}(x)}{C_{j}^{(d-1)/2}(1)} = \frac{j !}{(d-1)_{j}} C_{j}^{(d-1)/2}(x), \ \ \  x \in [-1, 1].
\end{equation}
It is straightforward to see that
\begin{equation}\label{Gegen-dim-eq}
	\frac{\dim (\mathcal{H}_{j}^{d})}{C_{j}^{(d-1)/2}(1)} = \frac{2j + d - 1}{d-1} \simeq  j.
\end{equation}

\begin{rem}
When $d = 2$, i.e., the $\mathbb{S}^{2}$ case, $C_{j}^{1/2} (\cdot) = P_{j} (\cdot) $, where $P_{j} (\cdot)$ is the Legendre polynomial of degree $j$ \cite{dai2013approximation}.
\end{rem} 

See \cite{dai2013approximation} and \cite{Morimoto-1998} for a deeper overview of spherical harmonics, and \cite{BE-1953} for a detailed description of Gegenbauer, Jacobi and Legendre polynomials.



\subsection{Isotropic Stationary Gaussian Random Fields on the Sphere cross Time}


	Let $(\Omega, \mathcal{F}, \mathbb{P})$ be a complete probability space. Consider $\mathbb{S}^{d}$ and $\mathbb{S}^{d}\times \mathbb{R}$ as manifolds contained in $\mathbb{R}^{d+1}$ and in $\mathbb{R}^{d+2}$ respectively. 


\begin{definition}
The geodesic distance $\theta$ on $\mathbb{S}^{d}$ (\textquotedblleft great circle\textquotedblright \ or \textquotedblleft spherical\textquotedblright \ distance) is defined by
\begin{equation}\label{GC-dist}
\displaystyle
	\theta ({\bf x}, {\bf y}) = \arccos(\langle {\bf x}, {\bf y} \rangle _{\mathbb{R}^{d+1}} ), \ \ \ {\bf x}, {\bf y} \in \mathbb{S}^{d}.
\end{equation}

The geodesic distance $\rho$ on $\mathbb{S}^{d}\times \mathbb{R}$ is defined through
\begin{equation*}\label{GCT-dist}
\displaystyle
	\rho(({\bf x}, s), ({\bf y},t)) = \sqrt{\theta \left({\bf x}, {\bf y} \right)^{2} + (t-s)^{2}}, \ \ \ ({\bf x}, s),  ({\bf y}, t) \in \mathbb{S}^{d} \times \mathbb{R}.
\end{equation*}
\end{definition}	
	
	Throughout, unless it is explicitly presented in a different way, we write $\mathbb{S}^{d}$ instead of $(\mathbb{S}^{d}, \sigma_{d})$, $\sigma_{d}$ being the surface area measure, which is equivalent to other uniformly distributed measures on $\mathbb{S}^{d}$, such as the Haar measure or the Lebesgue spherical measure \citep{MR0260979}. Analogously, we write $\mathbb{R}$ instead of $(\mathbb{R}, \mu)$, for $\mu$ being the Lebesgue measure.

\begin{definition}
A $\mathcal{F} \otimes \mathcal{B}(\mathbb{S}^{d} \times \mathbb{R})$-measurable mapping, $Z: \Omega \times \mathbb{S}^{d} \times \mathbb{R} \mapsto  \mathbb{C} $, is called a complex-valued random field on $\mathbb{S}^{d} \times \mathbb{R}$. 
\end{definition}

A complex-valued random field $Z = \left\lbrace Z(\mathbf{x}, t) :\ \mathbf{x} \in \mathbb{S}^{d},\ t \in \mathbb{R} \right\rbrace$ is called Gaussian if for all $k \in \mathbb{N}$, $(\mathbf{x}_{1}, t_{1}),\hdots,(\mathbf{x}_{k}, t_{k}) \in \mathbb{S}^{d} \times \mathbb{R}$, the random vector 
$$ \left( \Re Z(\mathbf{x}_{1}, t_{1}),\hdots, \Re Z(\mathbf{x}_{k}, t_{k}) , \Im Z(\mathbf{x}_{1}, t_{1}),\hdots, \Im Z(\mathbf{x}_{k}, t_{k}) \right)^{\top} , $$ is multivariate Gaussian distributed.
Here, $\top$ denotes the transpose operator.

A function $h:\mathbb{S}^{d} \times \mathbb{R} \times \mathbb{S}^{d} \times \mathbb{R} \longrightarrow \mathbb{C}$ is positive definite if 
\begin{equation}
\label{pos-def} \displaystyle
\sum_{i=1}^{n} \sum_{j=1}^{n} c_{i} \overline{ c_{j} } h(\mathbf{x}_{i}, t_{i}, \mathbf{x}_{j}, t_{j}) \geq 0,
\end{equation}
for all finite systems of pairwise distinct points $ \left\lbrace (\mathbf{x}_{k},t_{k}) \right\rbrace_{k = 1}^{n} \subset \mathbb{S}^{d} \times \mathbb{R}$ and constants $c_{1}, \ldots , c_{n} \in \mathbb{C}$.   A positive definite function $h$ is strictly positive definite if the inequality (\ref{pos-def}) is strict, unless $c_{1} =  \cdots = c_{n} = 0$. 

We call a function  $h:\mathbb{S}^{d} \times \mathbb{R} \times \mathbb{S}^{d} \times \mathbb{R} \longrightarrow \mathbb{C}$  spatially isotropic and temporally stationary if there exists a function $\psi : [-1, 1] \times \mathbb{R} \longrightarrow \mathbb{C}$ such that 
\begin{equation}
\label{Def_Iso-2}
h(\mathbf{x}, t, \mathbf{y}, s) = \psi \left(\cos \theta(\mathbf{x},\mathbf{y}), t-s \right), \ \ \ (\mathbf{x}, t), (\mathbf{y}, s) \in \mathbb{S}^{d} \times \mathbb{R}.
\end{equation}
Hence, a spatially isotropic and temporally stationary function depends on its arguments via the great circle distance $\theta(\mathbf{x}, \mathbf{y})$ and the time lag, or equivalently, via the inner product $\left\langle \mathbf{x}, \mathbf{y} \right\rangle_{\mathbb{R}^{d+1}} $ and the time lag.


\begin{definition}
\label{Def_Iso}
We call a random field $Z = \left\lbrace Z(\mathbf{x}, t) :\ \mathbf{x} \in \mathbb{S}^{d},\ t \in \mathbb{R}  \right\rbrace$ 2-weakly isotropic stationary if $\mathbb{E} \left[ Z(\mathbf{x}, t) \right]$ is constant for all $(\mathbf{x}, t) \in \mathbb{S}^{d} \times \mathbb{R}$, and if the covariance $h$ is a spatially isotropic and temporally stationary function on $(\mathbb{S}^{d} \times \mathbb{R})^{2}$. The associated function $\psi$ in (\ref{Def_Iso-2}) is called the covariance kernel or simply kernel. 
\end{definition}
 
\vspace{0.3cm} 
 
\begin{rem}
A Gaussian random field (GRF) which is  2-weakly isotropic stationary on $\mathbb{S}^{d} \times \mathbb{R}$, is in fact isotropic in the spatial variable and stationary in the time variable (see \cite{MP-2011}), hence, it has an invariant distribution under rotations on the spatial variable, and under translations on the temporal variable.  
\end{rem} 

Throughout the manuscript we work with zero-mean random fields, with no loss of generality. 

\subsection{Kernel Covariance Functions on the Sphere cross Time}

	In his seminal paper, \cite{schoenberg1942} characterized the class of continuous functions $f: [-1, 1] \mapsto  \mathbb{R}$ \ such that $f (\cos \theta)$ is positive definite over the product space $\mathbb{S}^{d} \times \mathbb{S}^{d}$, with $\theta$ defined in equation (\ref{GC-dist}).   Recently, \cite{Berg2016} extended Schoenberg's characterization by considering the product space $\mathbb{S}^{d} \times G$, with $G$ being a locally compact group. They defined the class ${\cal P}(\mathbb{S}^{d}, G)$ \ of continuous functions $ \psi: [-1, 1] \times G \mapsto  \mathbb{C}$ such that $ \psi (\cos \theta , u^{-1} \cdot v)$ is positive definite on $(\mathbb{S}^{d} \times G)^{2}$. In particular, the case $G=\mathbb{R}$  offers a characterization of spatio-temporal covariance functions of centred 2-weakly isotropic stationary random fields over the sphere cross time.

Let $\mathcal{P}(\mathbb{R})$ denote the set of continuous and positive definite functions on $\mathbb{R}$. For $d = 1,2, \ldots, $ we consider the class ${\cal P}(\mathbb{S}^{d}, \mathbb{R})$ of continuous functions $\psi:[-1,1] \times \mathbb{R}  \mapsto  \mathbb{\mathbb{C}}$ such that the associated spatially isotropic temporally stationary function $h$ in (\ref{Def_Iso-2}) is positive definite. The following result, rephrased from \cite{Berg2016}, allows to identify the class ${\cal P} (\mathbb{S}^{d}, \mathbb{R})$ with the covariances functions of 2-weakly isotropic stationary random fields on $\mathbb{S}^{d} \times \mathbb{R}$.

\begin{thm}
\label{Berg_Porcu_Thm}
Let $d \in \mathbb{N}$ and let $\psi:[-1,1] \times \mathbb{R} \mapsto  \mathbb{C}$ be a continuous mapping.   Then, $\psi \in {\cal P} (\mathbb{S}^{d}, \mathbb{R})$ if and only if there exists a sequence $\left\lbrace \varphi_{j,d} \right\rbrace_{ j \in \mathbb{N}} \in \mathcal{P}(\mathbb{R})$ with $\displaystyle \sum_{j=0}^{\infty} \varphi_{j,d}(0) < +\infty$ such that 
\begin{equation} \label{SBP}
\displaystyle
\psi \left( \cos\theta, t \right) = 
 \sum_{j=0}^{\infty} \varphi_{j,d}(t) c_{j}(d, \cos\theta ) \dim (\mathcal{H}_{ j }^{d}) , \ \ \  \theta \in [0,\pi], \ \ t \in \mathbb{R},
\end{equation}
where
\begin{equation}
\label{Sf}
\varphi_{j,d}(t) = \displaystyle \frac{ \| \sigma_{d} \|}{ \| \sigma_{d+1} \| } \int_{-1}^{1} \psi (x,t) c_{j}(d, x ) (1-x^{2})^{d/2 -1 } {\rm d}x, 
\end{equation}  
are called Schoenberg's functions, $\dim (\mathcal{H}_{ j }^{d})$ is given by (\ref{Gegen-norm}), and $c_{j}( \cdot, \cdot )$ by (\ref{SH_space_dim}). For $(\theta, t) \in [0,\pi] \times \mathbb{R}$, the series in (\ref{SBP}) is uniformly convergent. \\
\end{thm}


\begin{rem} Some comments are in order:
\begin{itemize}
	\item For ${\bm \varepsilon} _{1}  = (1,0, \ldots , 0)$ a unit vector in $\mathbb{R}^{d+1}$, we may consider the mapping $\psi_{ {\bm \varepsilon} _{1}} : \mathbb{S}^{d} \times \mathbb{R} \mapsto  \mathbb{C}$ given by	
\begin{equation*}
	\psi_{ {\bm \varepsilon} _{1}}({\bf y}, t) = \psi ( \left\langle {\bm \varepsilon} _{1}, {\bf y} \right\rangle_{\mathbb{R}^{d+1}}, t),  \ \ \  ({\bf y}, t) \in  \mathbb{S}^{d} \times  \mathbb{R}.
\end{equation*}
Then, arguments at page 13 of \cite{Berg2016}, show that
\begin{eqnarray*}
\displaystyle
	\varphi_{j,d}(t) &=&  \int_{\mathbb{S}^{d}} \psi (\left\langle {\bm \varepsilon} _{1}, {\bf y}\right\rangle _{\mathbb{R}^{d+1}} , t) c_{j}(d, \left\langle  {\bm \varepsilon} _{1}, {\bf y}\right\rangle _{\mathbb{R}^{d+1}}) \  {\rm d}\sigma_{d}({\bf y}) \\
	&=&  \left\langle \psi_{ {\bm \varepsilon} _{1}}(\cdot , t) , c_{j}(d, \left\langle {\bm \varepsilon} _{1}, \cdot \right\rangle _{\mathbb{R}^{d+1}} ) \right\rangle_{L^{2}(\mathbb{S}^{d} )}.
\end{eqnarray*}
Hence, for each $t \in \mathbb{R}$, the Schoenberg's functions $\varphi_{j,d}(t)$ given by (\ref{Sf}) may be understood as the orthogonal projection of  $\psi_{ {\bm \varepsilon} _{1}}(\cdot , t)$ onto  $\mathcal{H}_{j}^{d}$.


\item We note that, in comparison with the representation of covariance functions for 2-weakly isotropic random fields on the sphere,  representation (\ref{space_time_cov}) does not consider Schoenberg coefficients $\varphi_{j,d}$ but Schoenberg functions $\varphi_{j,d}( \cdot )$, which will play a fundamental role subsequently.

\end{itemize}

\end{rem}


\section{Expansions for Isotropic Stationary GRFs and Kernel Covariance Functions} \label{Expansions}

According to Theorem \ref{Berg_Porcu_Thm}, the kernel $\psi$ of an isotropic stationary GRF $Z$ on $\mathbb{S}^{d} \times \mathbb{R}$ admits the following representation:
\begin{equation}
\label{space_time_cov}
\psi (x, u) = \sum_{j=0}^{\infty} \varphi_{j,d}(u) c_{j}(d, x ) \dim (\mathcal{H}_{j}^{d}),  \ \ \  x \in [-1,1], \ \ u \in  \mathbb{R}, 
\end{equation}
where $ \left\lbrace \varphi_{j,d} \right\rbrace_{j \in \mathbb{N} } $ is a sequence of functions in $\mathcal{P}(\mathbb{R})$, such that the series is uniformly convergent. 


Expression (\ref{space_time_cov}) allows to consider different expansions for the kernel. Before introducing these representations, we present the expansion for the random field which motivates the simulation methodology.

\subsection{Karhunen-Loève Expansions for Isotropic Stationary GRFs on the Sphere cross Time}

	In \cite{Jones-1963}, the following Karhunen-Loève representation for an isotropic stationary GRF $Z$ over $\mathbb{S}^2 \times \mathbb{R}$ is proposed, 
\begin{equation}
\label{JR}
Z(\mathbf{x},t)   = \sum_{j=0}^\infty       \sum_{m = -j}^{j} X_{j,m}(t)\mathcal{Y}_{j,m}(\mathbf{x} ),  \ \ \ (\mathbf{x}, t) \in \mathbb{S}^2 \times \mathbb{R},
\end{equation}
where $ \left\lbrace X_{j,m}(t) \right\rbrace _{j, m} $ is a sequence of one-dimensional complex-valued mutually independent stochastic processes. The set of all $X_{j,m}(t)$ forms a denumerable infinite dimensional stochastic process which completely defines the process on the sphere, and $\mathcal{Y}_{j,m}(\mathbf{x} )$ are the elements of an orthonormal basis of $\mathcal{H}_{j}^{2}$.

Representations (\ref{space_time_cov}) for the spatio-temporal covariance and (\ref{JR}) for the random field, allow us to introduce the following family of GRFs on $\mathbb{S}^d \times \mathbb{R}$ for any $d \in \mathbb{N}$.

\begin{definition}
\label{Def_ACP}
Let $Z$ be a random field on $\mathbb{S}^d \times \mathbb{R}$ defined, in the mean square sense, as 
\begin{equation}
\label{sim_st_1}
Z(\mathbf{x},t)   = \sum_{j=0}^\infty  \sum_{m = 0}^{\dim (\mathcal{H}_{ j }^{d})} X_{j,m,d}(t)  \mathcal{Y}_{j,m,d}(\mathbf{x} ).
\end{equation}
Here, for each $j, m$ and $d$, $\{X_{j,m,d}(t), t\in\mathbb{R}\}$ is a complex-valued zero-mean stationary  Gaussian process such that 
\begin{equation*} \label{Cov_X}
\displaystyle
\begin{split} 
{\rm cov}\{X_{j,m,d}(t),X_{j ', m',d}(s)\} := & \  \mathbb{E} \{X_{j,m,d}(t) \ \overline{ X_{j ', m',d}(s) } \} 
\\
= & \ \varphi_{j,d}(t-s) \delta_{j , j '}\delta_{m , m'},
\end{split}
\end{equation*}
where $\{ \varphi_{j,d}\}_{j \in \mathbb{N} }$ represents the Schoenberg's functions associated to the mapping $\psi$ in Equation (\ref{space_time_cov}). 
\end{definition}

\begin{rem}
Examples of random field satisfying Definition (\ref{Def_ACP}) are found in the Appendix of \cite{BGP-2015}. 
\end{rem}

\begin{proposition}
\label{prop2}

Let $Z$ be a random field as in Definition \ref{Def_ACP}. 
Then, $Z$ is an isotropic stationary GRF,  with zero-mean and covariance function given by  $\psi$ as in (\ref{space_time_cov}).
\end{proposition}

\begin{proof} The proof follows straight by using the properties of the process $X_{j,m,d}$, and the addition formula in Equation (\ref{addition}). 
\end{proof}

	For each $j, m, d$, the process $\{X_{j,m,d}(t), t\in\mathbb{R}\}$ has the following Karhunen-Loève expansion (see Appendix \ref{Apendix_KL}),
\begin{equation}
\label{KL-2}
\displaystyle
	X_{j,m,d}(t) =  \sum_{k=0}^\infty \lambda_{j, k, d} \, \zeta_{k}(t),  \ \ \ t \in \mathbb{R},
\end{equation}
where for each $j, m, d \in \mathbb{N}$, $\{\lambda_{j, k, m,d}\}_{k}$ is a sequence of independent complex-valued random variables defined by 
\begin{equation*} \label{rv_kl_2}
	\lambda_{j, k, m,d} := \int_{\mathbb{R}} X_{j,m,d}(t) \overline{ \zeta_{k}(t) } {\rm d}\mu(t).
\end{equation*}
Also, $\lambda_{j, k, m,d} \sim \mathcal{N}(0, a_{j, k, d})$, where $\{ a_{j, k, d}\}_{k}$ and $\{\zeta_k\}_{k}$ are the eigenvalues and eigenfunctions (respectively) of the integral operator $\mathcal{K}_{\varphi} : L^{2}(\mathbb{R}) \mapsto  L^{2}(\mathbb{R})$ associated to $\varphi_{j, d}$, defined by
\begin{equation*}\label{Schoenberg_Op}
\mathcal{K}_{\varphi} (f) (t) := \int_{\mathbb{R}} \varphi_{j, d} (t-s) \delta_{j , j '} \delta_{m , m'} f(s)  {\rm d}\mu(s), \ f \in L^{2}(\mathbb{R})
\end{equation*}

\begin{rem} By (\ref{KL-2}), an alternative way to write (\ref{sim_st_1}) is:
\begin{equation} \label{double_kl}
  Z(\mathbf{x},t) = \sum_{j=0}^{\infty}  \sum_{m = 1 } ^{ \dim (\mathcal{H}_{ j }^{d})}    \sum_{k=0}^\infty \lambda_{j, k, m,d} \, \zeta_{k}(t)   \mathcal{Y}_{j,m,d}(\mathbf{x}),  \ \ \ ({\bf x}, t) \in \mathbb{S}^{d} \times \mathbb{R}.
\end{equation}
Expressions (\ref{sim_st_1}) or (\ref{double_kl}) represent a way to construct isotropic stationary GRFs on the sphere cross time, and suggest a spectral simulation method. However, it is not yet proved that any  isotropic stationary GRF on the sphere cross time can be written in this way.  
\end{rem}

\subsection{Double Karhunen-Loève Expansion of Kernel Covariance Functions}

	By stationarity of the process $X$ and Karhunen-Loève theorem we have,
	

\begin{equation*}  \label{eq_1}
\begin{split}
\varphi_{j,d} (u) \delta_{j,j'} \delta_{m,m'} =& \ {\rm cov}\{X_{j,m,d}(u),X_{j ', m',d}(0)\} 
\\
= & \ \sum_{k=0}^{\infty} a_{j, k, d} \, \zeta_k(u)\zeta_k(0), \ \ \ u \in \mathbb{R}.
\end{split}
\end{equation*}
Hence, a slight abuse of notation allows to reformulate the last expression to

\begin{equation*}  \label{eq_2}
\varphi_{j,d} (u) =  \sum_{k=0}^{\infty} a_{j, k, d} \, \zeta_k(u), \ \ \ u \in \mathbb{R}.
\end{equation*}

Therefore, the kernel covariance function $\psi$ in (\ref{space_time_cov}) also admits the expansion
\begin{equation}\label{Ang_spa_time_cov}
 \displaystyle
 \psi\left( x , t \right) = \sum_{j = 0}^{\infty} \dim (\mathcal{H}_{j}^{d})  \sum_{k=0}^{\infty} a_{j, k, d} \, \zeta_k(t) c_{j} \left( d , x \right) ,  \ \ \ (x , t) \in [-1,1] \times \mathbb{R}.
\end{equation}
Following \cite{LS-2015}, we call the series $\{ a_{j, k, d}\}_{j, k \in \mathbb{N}} \subset \mathbb{R} $ the {\it spatio-temporal angular power spectrum}. 

Theorem \ref{Berg_Porcu_Thm} implies that $ \left\lbrace \sum_{k=0}^{\infty}   a_{j, k, d} \, \zeta_{k}(\cdot) \right\rbrace_{j \in \mathbb{N} } $ is a sequence of continuous and positive definite functions. Further,  
\begin{equation*} \label{constr1}
\sum_{j=0}^{\infty} \sum_{k=0}^{\infty} a_{j, k, d} \, \zeta_{k} ( 0 ) < +\infty.
\end{equation*}

\subsection{Hermite Expansion of Kernel Covariance Functions}

It is well known that any $\varphi \in \mathcal{P}(\mathbb{R})$ satisfies $|\varphi(t) | \leq \varphi(0)$ (see \citep{Sasvari1994}). Therefore, $\varphi \in \mathcal{P}(\mathbb{R})$, ensures that $\varphi \in L^{2}(\mathbb{R}, \nu)$ for any finite measure $\nu$, in particular, any Gaussian measure.

Let $\nu$ be the standard Gaussian measure. As the Schoenberg's functions $\{ \varphi_{j,d}\}_{j \in \mathbb{N} }$ associated to $\psi$ in Equation (\ref{space_time_cov}) belong to the class $\mathcal{P}(\mathbb{R}) \subset L^{2}(\mathbb{R}, \nu)$, they can be expanded in terms of normalized Hermite polynomials.   For each $j, d \in \mathbb{N}$ there exist constants $ \left\lbrace \alpha_{j,k,d} \right\rbrace_{k} \in \mathbb{C}$ such that
\begin{equation*}\label{Her_exp_cov}
\displaystyle
\varphi_{j,d} (u) = \sum_{k = 0}^{\infty} \alpha_{j, k, d} H_{k}(u), \ \ \ u \in \mathbb{R}.
\end{equation*}
The series converges in $L^{2}(\mathbb{R}, \nu)$, and each $H_{k}$ is a normalized Hermite polynomial of degree $k$ given by
\begin{equation*}\label{Her_pol}
\displaystyle
	H_{k}(u) = \frac{(-1)^{k}}{ ( k! \sqrt{2\pi} )^{1/2} } e ^{\frac{u^{2}}{2}} \frac{{\rm d}^{k}}{{\rm d}u^{k}} e ^{\frac{{-u}^{2}}{2}}, \qquad k = 0,1,2, \ldots
\end{equation*}

Consequently, the kernel $\psi$ in (\ref{space_time_cov}) can be reformulated as 
\begin{equation}\label{Her_spa_time_cov}
\displaystyle
\begin{split}
\psi( x , t) = & \ \sum_{j=0}^{\infty} \dim (\mathcal{H}_{j}^{d}) \sum_{k = 0}^{\infty} \alpha_{j, k, d} H_{k}(t) c_{j}(d , x) ,  \ \ \  ( x , t) \in [-1, 1] \times \mathbb{R}.
\end{split}
\end{equation}
We call the series $\{ \alpha_{k, j, d }\}_{j, k \in \mathbb{N}} \subset \mathbb{C} $ the {\it spatio-temporal Hermite power spectrum}.




\section{Regularity Properties}\label{Reg_Prop}

	This section is devoted to study the behaviour of the  kernel covariance functions associated to an isotropic stationary GRF $Z$.   It will be shown that the regularity of such kernels is closely related to the decay of the Hermite power spectrum or the angular power spectrum.   Moreover, the latter characterizes also the $(K-J)$-term truncation of a GRF $Z$ as in Equation (\ref{double_kl}).   
	We recall that we have introduced two expansions for the kernel covariance function of an isotropic stationary GRF $Z$:  		
\begin{itemize}
\item[i.-]    The Spatio-temporal power spectrum, by using a double Kaurhunen-Loève expansion according to formula (\ref{Ang_spa_time_cov}), valid for the kernel covariance function of any isotropic stationary GRF as in (\ref{double_kl}) over 
$\left( \mathbb{S}^{d} \times \mathbb{R} , \sigma_{d} \otimes  \mu \right) $ with $\mu$ the Lebesgue measure.

\item[ii.- ]  The Hermite power spectrum, by using Hermite polynomials according to formula (\ref{Her_spa_time_cov}), valid for the kernel covariance function of any isotropic stationary GRF $Z$ over $\left( \mathbb{S}^{d} \times \mathbb{R} , \sigma_{d} \otimes  \nu \right) $ with $\nu$ the standard Gaussian measure.
\end{itemize}	

We recall that, for all $\mathbf{x}, \mathbf{y} \in \mathbb{S}^{d}$ and $t, s \in \mathbb{R}$,
\begin{equation*} \label{Cov_relations}
{\rm cov}\left( Z({\bf x},t) , Z({\bf y},s) \right) = \psi(\cos \theta(\mathbf{x},\mathbf{y}), t-s ) = \psi (\langle \mathbf{x},\mathbf{y}\rangle_{\mathbb{R}^{d+1}}, t-s).
\end{equation*}

\begin{rem} Considering the relations between Gegenbauer and Legendre polynomials, in the case $d = 2$, the kernel (\ref{space_time_cov}) turns out to be
\begin{equation*} \label{Cov_S2}
\sum_{j = 0}^\infty \varphi_{j, 2}(t-s) (2j + 1) P_{j} ( \langle \mathbf{x},\mathbf{y}\rangle_{\mathbb{R}^{3}} ) = \psi (\langle \mathbf{x},\mathbf{y}\rangle_{\mathbb{R}^{3}}, t-s),  \ \ \  \mathbf{x},\mathbf{y} \in \mathbb{S}^{2}, \ t, s \in \mathbb{R}.
\end{equation*} 
\end{rem}

In what follows we present the regularity analysis of the kernel in terms of the behaviour of the two proposed expansions (\ref{Ang_spa_time_cov}) and (\ref{Her_spa_time_cov}). We address first the relation with the Hermite power spectrum (\ref{Her_spa_time_cov}).

\subsection{Regularity analysis for the Hermite expansion} \label{Her_Reg_Ana}

	In this part of the manuscript we consider the measure space $\left( \mathbb{R}, \nu \right) $ with $\nu$ the standard Gaussian measure. 
	
	For $n, m \in \mathbb{N}$ let the spaces $ H^{n}(-1,1) \subset L^{2}(-1,1)$ and $ H^{m}(\mathbb{R}) \subset L^{2}(\mathbb{R})$ be the standard Sobolev spaces. We extend the proposal in \cite{LS-2015} and  consider the function spaces $W^{n, m}:= W^{n,m} ((-1,1) \times \mathbb{R})$ as the closure of $H^{n}(-1,1) \times H^{m}(\mathbb{R})$ with respect to the weighted norm $\| \cdot \|_{W^{n,m}((-1,1) \times \mathbb{R})}$ given by
\begin{equation*} \label{Sov-norm-Her}
\displaystyle
\| f \|_{W^{n,m}}^{2} := \sum_{j =0}^{n} \sum_{k =0}^{m}  | f |_{W^{j, k}}^{2},
\end{equation*}
where for $j, k \in \mathbb{N}$ 
\begin{equation*} \label{Sov-sem-norm-Her}
\displaystyle
|f |_{W^{j, k}} ^{2} := \int_{\mathbb{R}} \int_{-1}^{1} \bigg| \frac{\partial^{k}}{\partial t^{k}} \frac{\partial ^{j}}{\partial x ^{j}} f (x, t) \bigg| ^{2} (1 - x ^{2})^{d/2 - 1 + j} {\rm d}x\ {\rm d}\nu (t).
\end{equation*}

Note that $\left( W^{n,n}, n\in \mathbb{N} \right)$ is a decreasing scale of separable Hilbert spaces,  i.e.
$$
L^{2}\left( (-1,1) \times \mathbb{R} \right) \cong
L^{2}(-1,1) \otimes L^{2}(\mathbb{R}) = W^{0,0} \supset W^{1,1} \supset \ldots \supset W^{n,n} \supset \ldots 
$$

We abuse of notation by writing $L^{2}$ instead of $ L^{2}((-1, 1) \times \mathbb{R})$, and we consider the canonical partial order relation in $\mathbb{N}^{2}$: $ (n,m) \leq (n', m') $ if and only if $ [n \leq n' \ \  \mbox{and} \ \ m \leq m']$.  

By Theorem 5.2 in \cite{MR2424078}, the norm of $W^{n,n}$ is equivalent to the first and the last element of the sum, i.e. 
\begin{equation*} \label{equiv-norm}
\| f \|_{W^{n,n}}^{2} \simeq \| f \|_{L^{2}}^{2} + | f |_{W^{n,n}}^{2} \ , \ \ f \in W^{n,n}.
\end{equation*}
	
We now derive another equivalent norm of ${W^{n,n}}$ in terms of summability of the spectrum. We first observe that, as the normalized Hermite polynomials $\{H_{k}\}_{k \in \mathbb{N}}$ constitute a orthonormal basis of $L^2(\mathbb{R})$, and that for any fixed $r > -1/2$ the Gengenbauer polynomials  $\{C_{j}^{r} \}_{j \in \mathbb{N}}$ a basis for $L^2(-1, 1)$, it is apparent that $\{ H_{k} \cdot C_{j}^{r} \}_{j, k \in \mathbb{N}}$ is a basis for   $L^{2}((-1, 1) \times \mathbb{R}) \cong L^{2}(-1,1) \otimes L^{2}(\mathbb{R})$. Therefore, any $f \in L^{2}$ can be expanded in the series 

\begin{equation} \label{Her_Leg_exp}
\displaystyle
f (x, t) = \sum_{k, j = 0}^{\infty} b_{k,j} H_{k}(t) C^{r}_{j}(x), \ \ \ (x,t) \in [-1,1] \times \mathbb{R},
\end{equation}
with
\begin{equation*}
\displaystyle
	b_{k,j} := \int_{\mathbb{R}} \int_{-1}^{1} f (x, t) H_{k}(t)  C^{r}_{j}(x) {\rm d}x{\rm d}\nu (t) , \ \ \  k, j \in \mathbb{N}_{0}.
\end{equation*}
Putting $\alpha_{j, k} := \frac{2j + d - 1}{d-1} b_{k,j} $, we get 
\begin{equation*} 
\displaystyle
f (x, t) = \sum_{j = 0}^{\infty} \dim (\mathcal{H}_{j}^{d}) \sum_{k = 0}^{\infty} \alpha_{j, k} H_{k}(t) c_{j}(d,x) , \ \ \ (x,t) \in [-1,1] \times \mathbb{R},
\end{equation*}
that is, $f$ can be written as a covariance kernel $\psi$ of the type (\ref{Her_spa_time_cov}).


This allows to tackle the problem in a different way.  Instead of using spectral techniques, regularity of kernels might be shown through an isomorphism between the spaces $W^{n,m} ((-1, 1) \times \mathbb{R})$ and the weighted square summable bi-sequence spaces 
\begin{equation}\label{BiSeq_Spaces}
	(k j)_{m, n} := \ell ^{2} \left( k^{m/2} \cdot j^{(d-1)/2 + n} \ ; \ \ k, j \in \mathbb{N} \right),
\end{equation}
where $\left\lbrace k^{m/2} \cdot j^{(d-1)/2 + n}; \ \ k, j \in \mathbb{N} \right\rbrace$ denotes the sequence of weights. 

From now on, for the sake of simplicity, we only consider  weighted Sobolev spaces $W^{n} := W^{n,n}$, and the weighted square summable bi-sequence spaces $(kj)_{n} := (k j)_{n, n}$, obtained as a special case of Equation (\ref{BiSeq_Spaces}) when $n = m$.  


In order to extend the isomorphism to spaces $W^{\eta}$ with $\eta$ being not an integer, following \cite{triebel1999interpolation} we now introduce the interpolation spaces $W^{\eta}:= W^{\eta}((-1, 1) \times \mathbb{R})$ for $n < \eta < n+1$, defined through:
\begin{equation*} \label{Inter-space}
\displaystyle
	 W^{\eta} := \left( W^{n}, W^{n+1} \right)_{\eta - n, 2} \ ,
\end{equation*}
equipped with the norm $\| f \|_{W^{\eta}} $ given by
\begin{equation*} \label{norm_inter_space}
\displaystyle
	\| f \|_{W^{\eta}}^{2} = \int_{0}^{\infty} \zeta^{-2(\eta -n)} | J(r, f) |^{2} \frac{{\rm d} \zeta}{\zeta},
\end{equation*}
where the functional $J$ is defined by
\begin{equation*} \label{J}
\displaystyle
	J (r, f) = \inf_{\substack{ f = v + w  \\ v \in W^{n}, \ w \in W^{n+1} }}  \left( \| v \|_{W^{n}} + \zeta \cdot \| w \|_{W^{n+1}} \right), \ \ \ \zeta > 0. 
\end{equation*}


The definition of the interpolation spaces $(kj)_{\eta}$ for $\eta$ non-integer is carried out in analogous way. 

The interpolation property (see section 2.4.1 in \cite{triebel1983theory}), implies that, if the spaces $W^{n}$ and $(kj)_{n}$ are isomorphic for all $n \in \mathbb{N}$, then they are isomorphic for all $\eta \in \mathbb{R}_{+}$. 


\begin{thm} \label{Reg_Her_exp}
Let $f \in L^{2}$ having expansion (\ref{Her_Leg_exp}) and $\eta \in \mathbb{R}_{+}$ be given. Then, $f  \in W^{\eta}$ {\it if and only if} 
$$
\displaystyle
\sum_{k, j = 0}^{\infty} | b_{k, j} | ^{2} k^{\eta} \cdot j ^{d - 1 + 2\eta } \ \ < \infty, $$
 i.e.
$$
\displaystyle
\| f \|_{W^{\eta}}^{2} \simeq 
 \sum_{k, j = 0}^{\infty} | b_{k, j} | ^{2}  k^{\eta} \cdot j ^{d - 1 + 2\eta }
$$
is an equivalent norm in $ W^{\eta}$.

For $\psi \in \ W^{n}$, $n \in \mathbb{N}$, the equivalence is reduced to: $\left\lbrace \alpha_{k, j} k^{n/2} j^{(d-1)/2 + n },\  k, j \in \mathbb{N} \right\rbrace$ is in $\ell^{2}(\mathbb{N}^{2})$ {\it if and only if}  $ \frac{\partial ^{n}}{\partial t ^{n}} \frac{\partial ^{n}}{\partial x ^{n}} \psi(x, t) (1 - x ^{2})^{(d - 2)/4 + n/2}$  is in  $L^{2} \left( (-1,1) \times \mathbb{R} \right) $, where $\left\lbrace \alpha_{k, j} \right\rbrace_{k ,j \in \mathbb{N} }$ is the space-time Hermite power-spectrum. Rephrased, 
\begin{equation*}
\displaystyle
	\sum_{k, j = n}^{\infty} | \alpha_{k, j} | ^{2} \  k^{n} \cdot j ^{d - 1 + 2n} < +\infty
\end{equation*}
{\it if and only if}
\begin{equation*}
\displaystyle
	\int_{\mathbb{R}} \int_{-1}^{1} \bigg| \frac{\partial ^{n}}{\partial t ^{n}}  \frac{\partial ^{n}}{\partial x ^{n}} \psi (x, t) \bigg| ^{2} (1 - x ^{2})^{d/2 -1 + n } \ {\rm d}x\  {\rm d}\nu (t) < +\infty. 
\end{equation*} 
\end{thm}

\begin{proof}
Assume first that the claim is already proved for $\eta \in \mathbb{N}$, i.e., $W^{n}$ is isomorphic to the weighted bi-sequence space $(k j)_{n}$ for all $n \in \mathbb{N}$.

Fix $n \in \mathbb{N}$, let $n < \eta < n+1$ and set $\kappa := \eta - n$. By the interpolation theorem of Stein-Weiss (see Theorem 5.4.1 in \cite{MR0482275}), the weights of $(k j)_{\eta}$ are given by
\begin{equation*}
	\left( k^{n} j ^{d -1 + 2 n } \right)^{1-\kappa} \cdot \left(  k^{n+1} j ^{d - 1 + 2 (n + 1) } \right)^{\kappa} = k^{\eta} \ j^{d - 1 + 2\eta }.
\end{equation*}

Now, we prove the isomorphism between $W^{n}$ and $(k j)_{n}$ for $n \in \mathbb{N}$, which is equivalent to prove the second formulation of the theorem. We have that,
\begin{eqnarray} \label{KL-iso-1}
\displaystyle
\nonumber
	& & \int_{\mathbb{R}} \int_{-1}^{1} \bigg| \frac{\partial ^{n}}{\partial t ^{n}} \frac{\partial ^{n}}{\partial x ^{n}} \psi (x, t) \bigg| ^{2} (1 - x ^{2})^{d/2 - 1 + n } \ {\rm d} x\ {\rm d}\nu (t) \\
\nonumber	
    &=& \int_{\mathbb{R}} \int_{-1}^{1} \bigg| \sum_{j = 0}^{\infty} \dim (\mathcal{H}_{j}^{d}) \sum_{k = 0}^{\infty}  \alpha_{k, j} \frac{\partial ^{n}}{\partial t ^{n}} H_{k}(t) \frac{\partial ^{n}}{\partial x ^{n}} c_{j} (d , x)  \bigg| ^{2} (1 - x ^{2})^{d/2 - 1 + n } \ {\rm d}x \ {\rm d}\nu (t)  \\
\nonumber    
	&=&  \sum_{k, k', j, j ' = 0}^{\infty} \alpha_{k, j} \ \overline{ \alpha_{k', j '} } \int_{\mathbb{R}} \frac{\partial ^{n}}{\partial t ^{n}} H_{k}(t) \frac{\partial ^{n}}{\partial t ^{n}} H_{k'}(t) \ {\rm d}\nu (t) 
\\  \nonumber
	& & \ \cdot \int_{-1}^{1} \frac{\partial ^{n}}{\partial x ^{n}} c_{j} (d , x) \dim (\mathcal{H}_{j}^{d}) \frac{\partial ^{n}}{\partial x ^{n}} c_{j '}(d , x)  \dim (\mathcal{H}_{j'}^{d})  (1 - x ^{2})^{d/2 - 1 + n } \ {\rm d} x \\
	&=& \sum_{k, k', j, j ' = 0}^{\infty} \alpha_{k, j} \ \overline{ \alpha_{k', j '} } \cdot I_{k, k'} \cdot \widetilde{I}_{j, j '}  ,
\end{eqnarray}
where 
\begin{equation*}
	I_{k, k'} = \int_{\mathbb{R}} \frac{\partial ^{n}}{\partial t ^{n}} H_{k}(t) \frac{\partial ^{n}}{\partial t ^{n}} H_{k'}(t) \ {\rm d}\nu (t)  \ \ \ 
\end{equation*}
and
\begin{equation*}
	\widetilde{I}_{j, j '}  = \int_{-1}^{1} \frac{\partial ^{n}}{\partial x ^{n}} c_{j} (d , x) \dim (\mathcal{H}_{j}^{d}) \frac{\partial ^{n}}{\partial x ^{n}} c_{j '} (d , x) \dim (\mathcal{H}_{j'}^{d}) (1 - x ^{2})^{d/2 - 1 + n } \ {\rm d}x.
\end{equation*}

Standard properties of normalized Hermite polynomials show that $I_{k, k'} = 0$ if $k < n$, and 
\begin{equation*}
\displaystyle
	I_{k, k'} = \frac{k!}{(k-n)!} \delta_{k, k'}  \ ;\ \ k \geq n.
\end{equation*}
On one hand, Stirling inequality 
\begin{equation*}
	\sqrt{2\pi} \ k^{k + 1/2} \ e^{-k} \leq k! \leq  k^{k + 1/2} \ e^{-k + 1},
\end{equation*}
implies that
\begin{equation*}
	\sqrt{2\pi} \ e^{-(n+1)} \frac{k^{k + 1/2}}{(k-n)^{k-n + 1/2}} \leq \frac{k!}{(k-n)!} \leq \frac{e^{-n+1} }{\sqrt{2\pi}} \ \frac{k^{k + 1/2}}{(k-n)^{k-n + 1/2}}.
\end{equation*}
On the other hand,  
\begin{equation*}
	\frac{k^{k}}{(k-n)^{k-n}} = \frac{k^{n}}{(1 - n/k)^{k(1 - n/k)}},
\end{equation*}
where for $k \geq n$, 
\begin{equation*}
	\frac{1}{(1 - n/k)^{k - n}} \ \xrightarrow[k  \to \infty]{} e^{-n}.
\end{equation*}
Hence, for $k \geq n$, there exists constants $c_{1}(n)$ and $c_{2}(n)$ such that 
\begin{equation*}
	c_{1}(n) k^{n} \leq \frac{k!}{(k-n)!} \leq c_{2}(n) k^{n}.
\end{equation*}

Therefore, 
\begin{equation} \label{KL-iso-2}
\displaystyle
	I_{k, k'} \simeq k^{n} \delta_{k, k'}  \ ,\ \ k \geq n.
\end{equation}

For $I_{j, j '}(x)$, expressions (\ref{Gegen-ortho-2}), (\ref{Gegen-deriv}), (\ref{Gegen-norm}) and (\ref{Gegen-dim-eq}) allow to conclude that $I_{j, j '}(x) = 0$ if $j < n$, and 
\begin{equation} \label{KL-iso-3}
\displaystyle
	I_{j, j '}(x) \simeq j ^{d - 1 + 2n } \delta_{j, j '} \ ,\ \ j \geq n.
\end{equation}


Therefore, from (\ref{KL-iso-1}), (\ref{KL-iso-2}) and (\ref{KL-iso-3}) we deduce that 
\begin{eqnarray*}
\displaystyle
\begin{split}
& \int_{\mathbb{R}} \int_{-1}^{1} \bigg| \frac{\partial ^{n}}{\partial t ^{n}} \frac{\partial ^{n}}{\partial x ^{n}} \psi (x, t) \bigg| ^{2} (1 - x ^{2})^{d/2 - 1 + n } \ {\rm d}x\ {\rm d}\nu (t) \ 
\\
& \simeq \  \sum_{k, j = n}^{\infty} | \alpha_{k, j} | ^{2} \ k^{n} \ j^{d - 1 + 2n},
\end{split}
\end{eqnarray*}
which concludes the proof. 
\end{proof}

\subsection{Regularity analysis for the double Karhunen-Loève expansion}

	With the exception of minor details on the definitions of the spaces, this part of the manuscript follows similarly to Section \ref{Her_Reg_Ana}. On the other hand, we now consider the measure space $\left( \mathbb{R}, \mu \right) $ with $\mu$ the Lebesgue measure. 


For $n \in \mathbb{N}$ consider the function spaces $V_{T}^{n}:= V_{T}^{n} ((-1, 1)\times \mathbb{R})$ as the closure of $H^{n}(-1,1) \times L^{2}(\mathbb{R})$ with respect to the weighted norms $\| \cdot \|_{V_{T}^{n}}$ given by
\begin{equation*} \label{Sov-norm}
\displaystyle
\| f \|_{V_{T}^{n} ((-1, 1)\times \mathbb{R})}^{2} := \sum_{j=0}^{n} | f |_{V_{T}^{j} ((-1, 1)\times \mathbb{R})}^{2},
\end{equation*}
where for $j \in \mathbb{N}$ 
\begin{equation*} \label{Sov-sem-norm}
\displaystyle
|f |_{V_{T}^{j}} ^{2} := \int_{\mathbb{R}} \int_{-1}^{1} \bigg| \frac{\partial ^{j}}{\partial x ^{j}} f (x, t) \bigg| ^{2} (1 - x^{2})^{d/2 - 1 + j } {\rm d}x\  {\rm d}\mu (t) .
\end{equation*}
$\left( V_{T}^{n} \right)_{ n\in \mathbb{N}}$ is a decreasing sequence of separable Hilbert spaces, and  
\begin{equation*} \label{equiv-norm-1}
\displaystyle
\| f \|_{V_{T}^{n}}^{2} \simeq \| f \|_{L^{2}}^{2} + | f |_{V_{T}^{n}}^{2}, \ \ \ f \in V_{T}^{n}.
\end{equation*}


Now we look at the weighted square summable bi-sequence spaces 
$$
(kj)_{T, n} := \ell ^{2} \left( \left( j^{(d-1)/2 + n }; \ \ k, j \in \mathbb{N} \right) \right).
$$

As in Section \ref{Her_Reg_Ana}, we consider the interpolation spaces $V_{T}^{\eta}$ for $n < \eta < n+1$. The proof of next result follows exactly the same lines as the Theorem \ref{Reg_Her_exp}, hence it is omitted.

\begin{thm}
Let $f \in L^{2}$ and $\eta \in \mathbb{R}_{+}$ be given.   Then, $f  \in V_{T}^{\eta}$ if and only if 
$$
\displaystyle
\sum_{k, j = 0}^{\infty} b_{k, j}^{2} j^{d - 1 + 2\eta} < \infty, $$
i.e.
$$
\displaystyle
\| f \|_{V_{T}^{\eta}}^{2} \simeq \sum_{k, j = 0}^{\infty} b_{k, j}^{2} j^{d - 1 + 2\eta}
$$
is an equivalent norm in $ V_{T}^{\eta}$.

For $\psi \in \ V_{T}^{n}$, $n \in \mathbb{N}$,  the equivalence is reduced to: $\left\lbrace a_{k, j} \ j^{(d-1)/2 + n},\  k, j \in \mathbb{N} \right\rbrace$ is in $\ell^{2}(\mathbb{N}^{2})$ if and only if $ \frac{\partial ^{n}}{\partial x ^{n}} \psi(x, t) (1 - x ^{2})^{(d-2)/4 + n/2}$ is in  $L^{2}( (-1,1) \times \mathbb{R} )$, where $\left\lbrace a_{k, j} \right\rbrace_{k , j \in \mathbb{N} }$ is the space-time angular power-spectrum. Rephrased,
\begin{equation*}
\displaystyle
	\sum_{j = n}^{\infty} \sum_{k = 0}^{\infty} a_{k, j}^{2} \ j^{d - 1 + 2n}  < +\infty
\end{equation*}
if and only if
\begin{equation*}
\displaystyle
	\int_{\mathbb{R}} \int_{-1}^{1} \bigg| \frac{\partial ^{n}}{\partial x ^{n}} \psi (x, t) \bigg| ^{2} (1 - x ^{2})^{d/2 - 1 + n } {\rm d}x\  {\rm d}\nu (t) < +\infty. 
\end{equation*} \\
\end{thm}

\begin{rem}

	By taking into account the normalizing constants, all the previous results encompasses the results in Section 3 of \cite{LS-2015}.

\end{rem}





\section{Spatio-Temporal Spectral Simulation}
\label{Simul}

We now study a spectral simulation method for random fields on $\mathbb{S}^2\times [0,T]$, where $T$ denotes the time horizon. For a neater exposition, along this section, we omit the subscripts associated to the spatial dimension $d=2$.  

We must  first introduce some notation. Let $j\in\mathbb{N}_0$ and $m\in\{0,\hdots,j\}$. For $ x \in [-1,1]$, the associated Legendre polynomials $P_{j,m}$  are defined through
\begin{equation*}
 \label{legendre_functions}
\begin{split}
P_{j,m} ( x ) \ =& \   (-1)^m (1 - x ^2)^{m/2} \frac{\text{d}^m}{\text{d} x^m}(P_j( x )), 
\\
	=& \ (2m - 1)!! (-1)^{m} (1 - x^2)^{m/2} C_{j - m}^{m + 1/2}( x ).
\end{split}
\end{equation*}
The spherical harmonic basis functions,  $\mathcal{Y}_{j,m}: \mathbb{S}^2\mapsto  \mathbb{C}$,  are defined by
\begin{equation*} \label{spherical_harmonics}
\begin{split}
\displaystyle 
\mathcal{Y}_{j,m} (\textbf{x}) \ =&  \  \sqrt{\frac{2j +1}{4\pi} \frac{(j-m)!}{(j+m)!}} P_{j,m}(\cos \beta_1) \exp(im \beta_2),  \quad  j\in\mathbb{N}_0, m\in\{0,\hdots,j\} 
\\
\displaystyle 
\mathcal{Y}_{j,m} (\textbf{x}) \ =& \   (-1)^m \overline{\mathcal{Y}_{j,-m}}(\textbf{x}),    \quad       j\in\mathbb{N}, m\in\{-j,\hdots,-1\}.
\end{split}
\end{equation*}
where  $(\beta_1,\beta_2) \in [0,\pi]\times[0,2\pi)$ represents the spherical coordinates of $\textbf{x}\in\mathbb{S}^2$. 

On the other hand, let  $\{X_{j,m}(t), j\in\mathbb{N}_0, m\in\{-j,\hdots,j\} \}$ be a collection of  stochastic processes.  Thus, we consider the space-time random field
\begin{equation} 
\label{double_kl_1}
 {Z}(\mathbf{x},t)    =      \sum_{j=0}^{\infty}    \sum_{m=-j}^{j}   X_{j,m}(t)   \mathcal{Y}_{j,m}(\mathbf{x}), \qquad \mathbf{x}\in\mathbb{S}^2,t\in[0,T].
   \end{equation}

In order to obtain a real-valued field, we must impose some conditions on the stochastic processes $\{X_{j,m}(t), j\in\mathbb{N}_0, m\in\{0,\hdots,j\} \}$. Throughout, we assume that   $\{X_{j,m}(t), j\in\mathbb{N}_0, m\in\{0,\hdots,j\} \}$ are mutually independent,  with  $\Im(X_{j,0}(t))$ being  identically equal to zero, and  for $j\in\mathbb{N}$ and $m\in\{-j,\hdots,-1\}$,  
\begin{equation}
\label{condicion}
X_{j,m}(t) = (-1)^m \overline{X_{j,-m}}(t).
\end{equation}

Note that, using standard algebra of complex numbers, coupled with condition (\ref{condicion}), Equation (\ref{double_kl_1}) can be written as
\begin{eqnarray} \label{eq_RF}
\displaystyle
\begin{split}
 Z(\mathbf{x},t)  =&  \sum_{j=0}^{\infty}         \Bigg( X_{j,0}(t)   \mathcal{Y}_{j,0}(\textbf{x})  \\
	& +   \sum_{m=1}^j  \bigg\{ X_{j,m}(t)   \mathcal{Y}_{j,m} (\textbf{x})   + (-1)^m\overline{X_{j,m}}(t)  (-1)^m \overline{\mathcal{Y}_{j,m}}(\textbf{x})  \bigg\}       \Bigg) \\
   = & \sum_{j=0}^{\infty} \Bigg(   X_{j,0}(t)   \mathcal{Y}_{j,0}(\textbf{x})  \\
	& +  2 \sum_{m=1}^j  \bigg\{\Re( X_{j,m}(t))   \Re(\mathcal{Y}_{j,m} (\textbf{x}))  - \Im( X_{j,m}(t))   \Im(\mathcal{Y}_{j,m} (\textbf{x})) \bigg\} 
     \Bigg).
\end{split}
\end{eqnarray}   

We  consider  $\Re(X_{j,m}(t))$ and $\Im(X_{j,m}(t))$ as independent processes with the following Fourier expansions 
\begin{eqnarray*}
   \Re(X_{j,m}(t))   &  =  & A^1_{j, 0, m} + \sum_{k=1}^\infty   \left\{  A^1_{j, k, m} \cos\left(\frac{\pi k t}{2T}\right) +   B^1_{j, k, m}  \sin\left(\frac{\pi k t}{2T}\right) \right\} , \\
      \Im(X_{j,m}(t))   &  =  & A^2_{j, 0, m} + \sum_{k=1}^\infty   \left\{ A^2_{j, k, m}  \cos\left(\frac{\pi k t}{2T}\right) +   B^2_{j, k, m}  \sin\left(\frac{\pi k t}{2T}\right)  \right\}.
   \end{eqnarray*}
Here,  $\{A_{j, k, m}^q\}$ and $\{B_{j, k, m}^q\}$, for $q=1,2$, are  sequences of  independent centred real-valued Gaussian random variables, such that 
\begin{equation*}
{\rm var} (A^1_{j, k, 0}) = {\rm var}( B^1_{j, k, 0} )= a_{j, k}
\end{equation*}
 and    
 \begin{equation*}
 {\rm var} (A^q_{j, k, m}) = {\rm var} ( B^q_{j, k, m})  =   a_{j, k}/2, \quad \text{ for } m\neq 0,
 \end{equation*} 
 where $\{a_{j, k}\}_{k , j \in \mathbb{N} }$ is a summable  bi-sequence of non-negative coefficients. A direct calculation shows that the covariance function of $Z(\textbf{x},t)$ is  spatially isotropic and temporally stationary. More precisely, we have that
$$  {\rm cov}\{Z(\textbf{x},t), Z(\textbf{y}, s)\} = \sum_{j=0}^\infty \sum_{k=0}^\infty  \frac{ ( 2j+1 ) }{4\pi}   a_{j, k} \cos\left(\frac{\pi k (t - s )}{2T} \right) c_{j}(\langle \textbf{x}, \textbf{y} \rangle _{\mathbb{R}^3} ), \  \textbf{x}, \textbf{y}  \in \mathbb{S}^{2}. $$

Finally, given  two positive integers $J$ and $K$,  we truncate expression (\ref{eq_RF}) in the index $j$ and $k$, respectively. Thus, we  simulate a space-time Gaussian random field on $\mathbb{S}^2\times[0,T]$  using the  explicit approximation
\begin{eqnarray}\label{trunc}
\displaystyle
\begin{split}
 & \widehat{Z}(\mathbf{x},t)  =  \sum_{j=0}^{J}         \Bigg( A^1_{j, 0, 0}  \widetilde{P}_{j,0}(\cos \beta_1)
  \\
	+ &    2\sum_{m=1}^j   \widetilde{P}_{j,m}(\cos \beta_1)\bigg\{ A_{j, 0, m}^1     \cos(m\beta_2)  -     A_{j, 0, m}^2    \sin(m\beta_2)  \bigg\}  \\
	+ & \widetilde{P}_{j,0}(\cos \beta_1)  \sum_{k=1}^K   \bigg\{ A^1_{j, k, 0} \cos\left(\frac{\pi k t}{2T}\right)  + B^1_{j , k, 0} \sin\left(\frac{\pi k t}{2T}\right)  \bigg\}  \\
	+ & 2\sum_{m=1}^j     \widetilde{P}_{j,m}(\cos \beta_1) \cos(m\beta_2)  \sum_{k=1}^K \bigg\{   
A^1_{j, k, m} \cos\left(\frac{\pi k t}{2T} \right)         +  B^1_{j, k, m} \sin\left(\frac{\pi k t}{2T}\right)   \bigg\}    \\
	- & 2\sum_{m=1}^j     \widetilde{P}_{j,m}(\cos \beta_1) \sin(m\beta_2)  \sum_{k=1}^K \bigg\{   
A^2_{j, k, m} \cos\left(\frac{\pi k t}{2T} \right)         +  B^2_{j, k, m} \sin\left(\frac{\pi k t}{2T}\right)   \bigg\}    
 \Bigg),
 \end{split}
\end{eqnarray}   

where  $\widetilde{P}_{j,m}(\cdot) = \sqrt{(2j+1)(j-m)! /  (4\pi(j+m))! } P_{j,m}(\cdot)$. 


	We assess the $L^2(\Omega\times\mathbb{S}^2\times[0,T])$ error associated to  the truncated expansion $\widehat{Z}(\textbf{x},t)$ given in Equation  (\ref{trunc}), in terms of the positive integers $J$ and $K$.   We follow the scheme used by \cite{LS-2015} in the spatial context and extend their result to the space-time case.    Next, we state the main result of this section.    

\begin{thm}
\label{accurate}
Let $\nu_1, \nu_2 \geq  2$.  Suppose that there exist positive constants $C_i$, for $i=1,2,3$,  and positive integers $j_0$ and $k_0$, such that  $a_{j, 0} \leq C_1 j^{-\nu_1}$, $a_{0, k} \leq C_2 k^{-\nu_2}$ and  $a_{j, k} \leq C_3  j^{-\nu_1} k^{-\nu_2}$, for all $j\geq j_0$ and $k\geq k_0$. Then,  the following inequality holds
\begin{equation}
\label{bound_error}
 \|   Z(\mathbf{x},t) - \widehat{Z}(\mathbf{x},t) \|^2_{L^2(\Omega\times\mathbb{S}^2\times[0,T])}  \leq \widetilde{C}_1 J^{-(\nu_1-2)} + \widetilde{C}_2 J K^{-(\nu_2-1)},
\end{equation}
for some positive constants $\widetilde{C}_1$ and $\widetilde{C}_2$.
\end{thm}

\begin{proof}

We  decompose ${Z}(\mathbf{x},t)-\widehat{Z}(\mathbf{x},t)$ in Equation (\ref{trunc}) into two independent terms, $${Z}(\mathbf{x},t)-\widehat{Z}(\mathbf{x},t) = T_1(\mathbf{x}) + T_2(\mathbf{x},t),$$
 where
\begin{eqnarray*} \label{termino1}
\displaystyle
\begin{split}
T_1(\mathbf{x})  = &   \bigg( \sum_{j=J+1}^{\infty}    A^1_{j, 0, 0}  \widetilde{P}_{j,0}(\cos \beta_1) \\
& +    2\sum_{m=1}^j   \widetilde{P}_{j,m}(\cos \beta_1)\bigg\{ A_{j, 0, m}^1     \cos(m\beta_2)  -     A_{j, 0, m}^2    \sin(m\beta_2)  \bigg\} \bigg)
\end{split}
\end{eqnarray*}
and 
\begin{equation*} \label{termino2}
\begin{split}
T_2(\mathbf{x},t) =  
\sum_{j=J+1}^\infty \sum_{k=1}^\infty  \Delta_{j, k}(\mathbf{x},t)  +   \sum_{j=0}^J \sum_{k=K+1}^\infty \Delta_{j, k}(\mathbf{x},t),
 \end{split}
\end{equation*}   
with  $\Delta_{j, k}(\mathbf{x},t)$ defined as
\begin{eqnarray*}
\displaystyle
\begin{split}
 & \Delta_{j, k}(\mathbf{x},t)  =   
\widetilde{P}_{j,0}(\cos \beta_1)      \bigg\{ A^1_{j, k, 0} \cos\left(\frac{\pi k t}{2T}\right)  + B^1_{j, k, 0} \sin\left(\frac{\pi k t}{2T}\right)  \bigg\}  \\
& + 2\sum_{m=1}^j     \widetilde{P}_{j,m}(\cos \beta_1) \cos(m\beta_2)  \bigg\{   
A^1_{j, k, m} \cos\left(\frac{\pi k t}{2T} \right)        +  B^1_{j, k, m} \sin\left(\frac{\pi k t}{2T}\right)   \bigg\}    \\
& - 2\sum_{m=1}^j     \widetilde{P}_{j,m}(\cos \beta_1) \sin(m\beta_2)   \bigg\{   
A^2_{j, k, m} \cos\left(\frac{\pi k t}{2T} \right)        +  B^2_{j, k, m} \sin\left(\frac{\pi k t}{2T}\right) \bigg\} .
\end{split}   
\end{eqnarray*}   
For the second term, we have used the identity
$$  \sum_{j=0}^\infty \sum_{k=1}^\infty   \xi_{j, k} -   \sum_{j=0}^J \sum_{k=1}^K \xi_{j, k}  =      \sum_{j=J+1}^\infty \sum_{k=1}^\infty  \xi_{j, k} +   \sum_{j=0}^J \sum_{k=K+1}^\infty \xi_{j, k},$$ 
which is satisfied for any  summable bi-sequence $\left\lbrace \xi_{j, k} \right\rbrace_{j , k \in \mathbb{N} }$.

The independence of $T_1(\mathbf{x})$ and $T_2(\mathbf{x},t)$ implies that 
$$\|   Z(\mathbf{x},t) - \widehat{Z}(\mathbf{x},t) \|^2_{L^2(\Omega\times\mathbb{S}^2\times[0,T])}   =   \|   T_1(\mathbf{x})  \|^2_{L^2(\Omega\times\mathbb{S}^2)} +   \|   T_2(\mathbf{x},t) \|^2_{L^2(\Omega\times\mathbb{S}^2\times[0,T])}.$$

In \cite{LS-2015} it is shown that there exists a positive constant $L_1$, depending on $C_1$, such that $$ \|   T_1(\mathbf{x})  \|^2_{L^2(\Omega\times\mathbb{S}^2)}   \leq  L_1 J^{-(\nu_1-2)}.$$

On the other hand, since $$\|  \cos(\pi k t/(2T))   \|^2_{L^2([0,T])}  + \|   \sin(\pi k t/(2T))    \|^2_{L^2([0,T])} = T,$$
 $$\|  \widetilde{P}_{j,0}(\cos \beta_1)     \|^2_{L^2(\mathbb{S}^2)}  = 1$$ 
 and 
 $$\|   \widetilde{P}_{j,m}(\cos \beta_1) \cos(m\beta_2)   \|^2_{L^2(\mathbb{S}^2)} + \|   \widetilde{P}_{j,m}(\sin \beta_1) \cos(m\beta_2)  \|^2_{L^2(\mathbb{S}^2)} = 1,$$
we have that  $\| \Delta_{j, k}(\mathbf{x},t)    \|^2_{L^2(\Omega\times\mathbb{S}^2\times[0,T])}= T (2j+1)a_{j, k}$.  Therefore,
\begin{eqnarray*}
\begin{split}
 \|   T_2(\mathbf{x},t) \|^2_{L^2(\Omega\times \mathbb{S}^2\times[0,T])}   \leq \, &  T \bigg\{  \sum_{j=J+1}^\infty (2j+1) \sum_{k=1}^\infty a_{j, k} \\
	&  + \sum_{k=K+1}^\infty    \bigg( a_{0, k}  +   \sum_{j=1}^J  (2j+1)a_{j, k}  \bigg)  \bigg\} \\
	\leq \, &  T \bigg\{  \bigg( C_3 \sum_{k=1}^\infty  k^{-\nu_2} \bigg) \bigg( \sum_{j=J+1}^\infty (2j+1) j^{-\nu_1} \bigg)  \\
    &  +  \bigg(  C_2  + C_3   \sum_{j=1}^J (2j+1) j^{-\nu_1}  \bigg)  \bigg(   \sum_{k=K+1}^\infty  k^{-\nu_2}  \bigg) \bigg\}\\
    \leq \, &  L_2 J^{-(\nu_1-2)} +   L_3  J  K^{-(\nu_2-1)},
\end{split}
\end{eqnarray*}
 where   $L_2$ and $L_3$ are positive constants depending on $C_2$ and $C_3$. In particular, the last inequality follows from the integral bounds of the corresponding series (see \cite{LS-2015}):
 \begin{eqnarray*}
 \sum_{j=J+1}^\infty (2j+1) j^{-\nu_1}   &  \leq  &   \left(  \frac{2}{\nu_1-2}  +  \frac{1}{\nu_1-1}\right) J^{-(\nu_1-2)}   \\
  \sum_{k=K+1}^\infty k^{-\nu_2}   &  \leq  &   \frac{1}{\nu_2-1}   K^{-(\nu_2-1)}.
 \end{eqnarray*}  
 The proof is completed.
\end{proof}

Simple examples can be generated from the following space-time angular power spectrum 
\begin{equation} \label{coef}
a_{j, k} = \frac{1}{1+ (1+j)^{\nu_1}  (1+k)^{\nu_2}},
\end{equation}
with  $\nu_{i} >2$, for $i =1,2$.        We illustrate  space-time realizations on $\mathbb{S}^2\times\{1,2\}$, over 24000 spatial locations, with  coefficients   (\ref{coef}), in two cases:   
\begin{itemize}
\item[(a)] $\nu_1=3$ and $\nu_2=5$, and
\item[(b)] $\nu_1=\nu_2=5$.
\end{itemize}
Figures \ref{rea1} and \ref{rea2} show the corresponding realizations for Scenarios (a) and (b), respectively.  For each case, we truncate the series using $K=J=50$. Note that the parameter $\nu_1$ is the responsible of the spatial scale and smoothness of the realization.   In \cite{LS-2015}, some realizations are illustrated using a similar spectrum,  in a merely  spatial context. 

   We now compare the empirical and theoretical  convergence rates for the cases (a) and (b) described above.  In our experiment, we consider $J=K$ and  study the  (Log) error in terms of  (Log) $J$, taking as the exact solution $J=50$.  Note that,   under this choice,  the bound (\ref{bound_error}) implies that the order of convergence     is $\min\{ (\nu_1-2) /2, (\nu_2-2) /2\}$. Following \cite{LS-2015}, instead of calculating the $L^2$-error, we take the maximum error over all the points on the space-time grid.   The empirical errors are calculated on the basis of 100 independent samples.  Our studies  reflect the theoretical results (see Figure \ref{error}).


\begin{figure}[h!]
\centering
\includegraphics[scale=0.24]{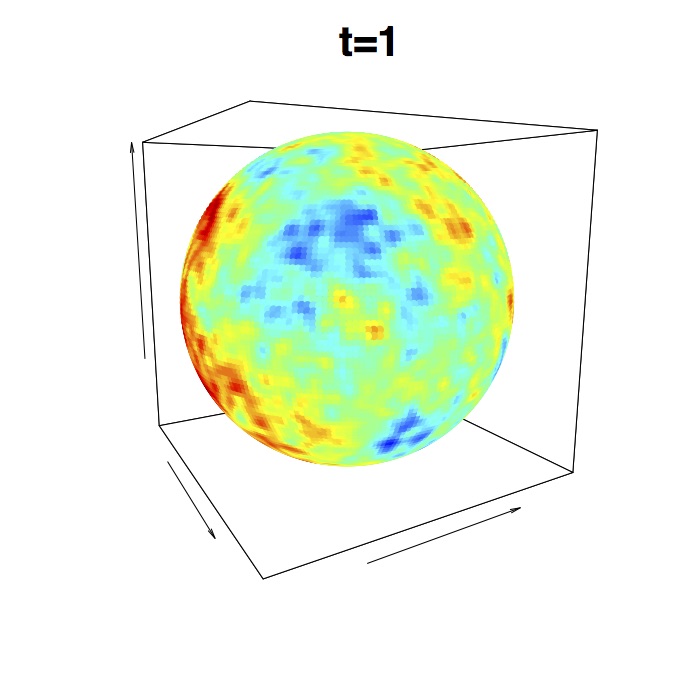} \includegraphics[scale=0.24]{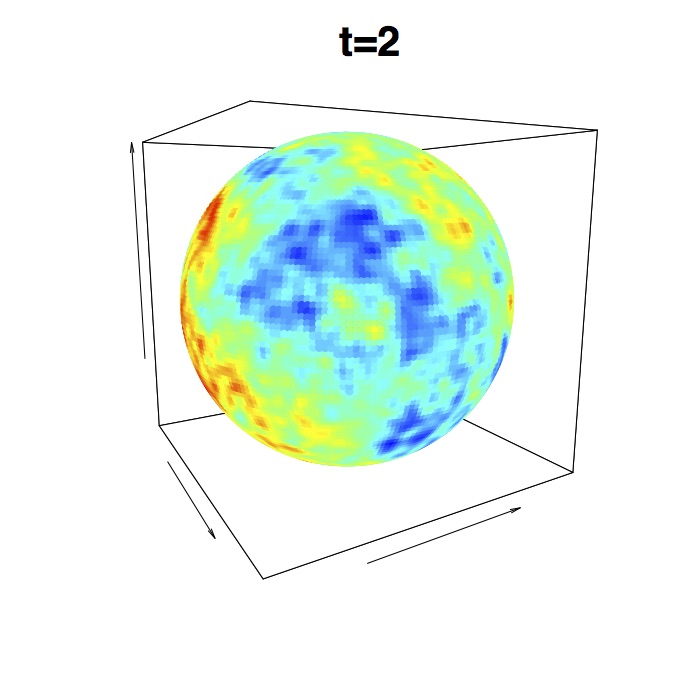}

\caption{Space-time realization on $\mathbb{S}^2\times \{1,2\}$, with spectrum  (\ref{coef}), with $\nu_1=3$ and $\nu_2=5$.}
\label{rea1}
\end{figure}

\begin{figure}[h!]
\centering
\includegraphics[scale=0.24]{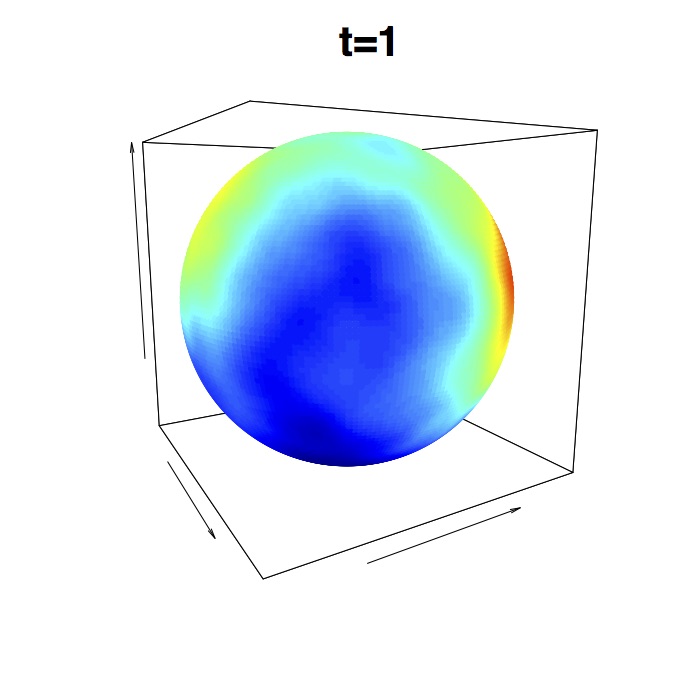} \includegraphics[scale=0.24]{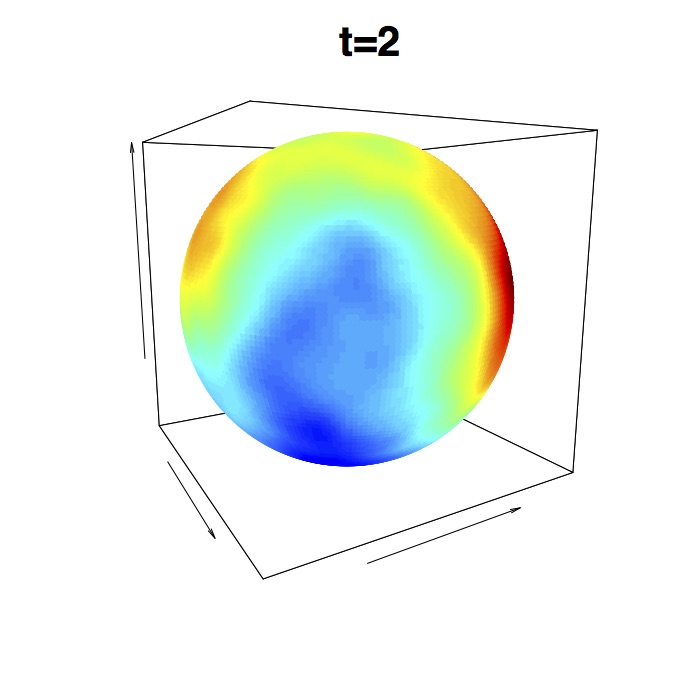}

\caption{Space-time realization on $\mathbb{S}^2\times \{1,2\}$, with spectrum  (\ref{coef}), with $\nu_1=\nu_2=5$.}
\label{rea2}
\end{figure}

\clearpage

\begin{figure}[h!]
\centering
\includegraphics[scale=0.24]{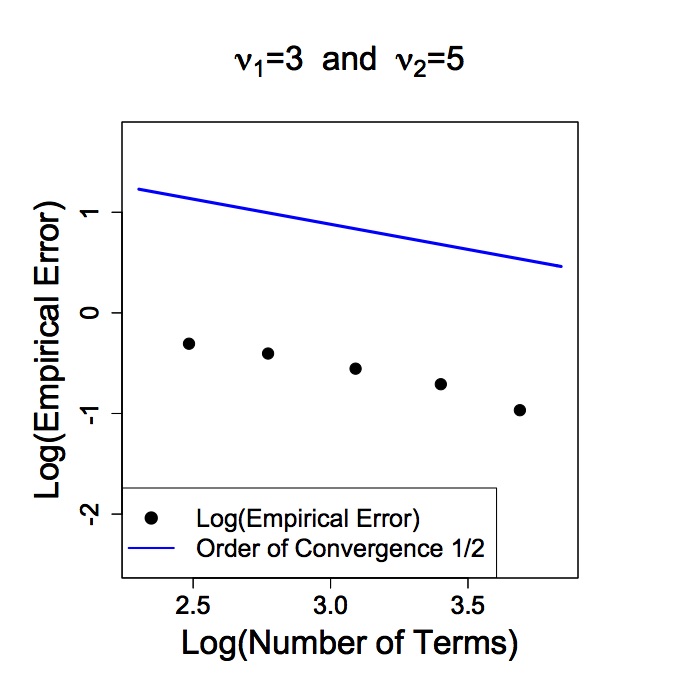} \includegraphics[scale=0.24]{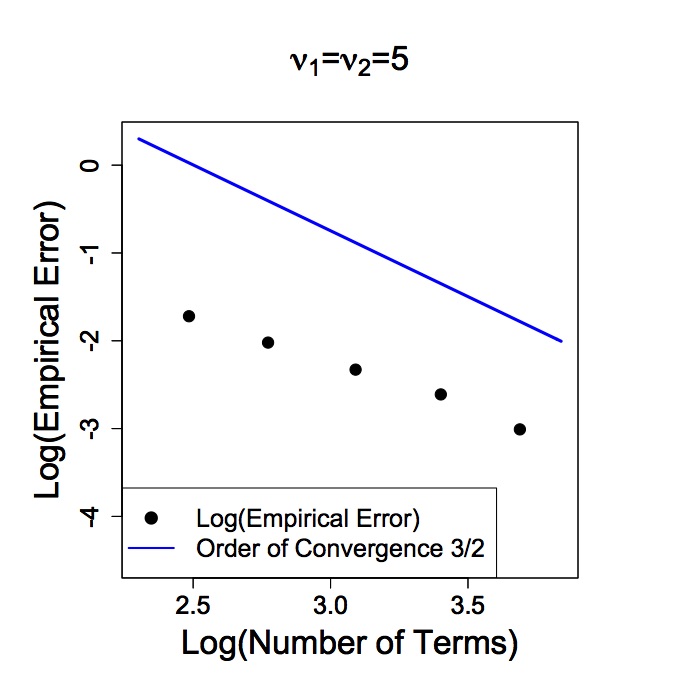}

\caption{Empirical versus theoretical (Log) $L^2$-error of the simulation method, in terms of   (Log) $J$,  with $K=J$. We consider  the spectrum (\ref{coef}) in two cases: (a) $\nu_1=3$ and $\nu_2=5$, and (b) $\nu_1=\nu_2=5$.}
\label{error}
\end{figure}


\section{Conclusions and discussion}

The present work has provided a deep look at the regularity properties of Gaussian fields evolving temporally over spheres. We hope this effort will put the basis for facing important challenges related to space-time processes. There are in fact many open problems related to mathematical modeling as well as to statistical inference and to optimal prediction. A list of open problems is included in the recent survey \cite{PAF-2017}. Amongst them, our paper is certainly related to Problem 1, that is to the construction of non-stationary processes on spheres cross time. Our work could also put the basis to solve Problem 2, related to the construction of multivariate space-time processes. It might be interesting to extend the study of regularity properties to the vector valued case. This would imply the use of a pretty different machinery. 
Problem 10 is closely related to our approach, because regularity properties are crucial to study Gaussian fields under infill asymptotics. 

On the other hand, a question arise naturally: is it possible to make inference with a representation like (\ref{double_kl}) ? Or with its respective spectral decomposition? The answer is, {\it a priori}, no. Establishing a clear relation between the parameters of any random field and its spectrum is not an obvious task, in fact, up today the only familiar stochastic process with known spectrum is Brownian motion, and his closer generalization, the fractional Brownian motion, doesn't have yet a known spectrum. However, under relatively weak hypotheses, the covariance kernel of a GRF turns out to be a Mercer kernel. This opens an alternative to the negative answer previously mentioned, by considering the eigenvalue problem associated to the integral operator induced by the kernel.


\section{Appendix}

\subsection{Karhunen-Loève Theorem} \label{Apendix_KL}

\hspace*{0.7cm} Recent results on functional analysis (see \cite{FM-2009} and \cite{FM-2013}) allow to construct Mercer's kernels in more general contexts. A proper interpretation of these results allows to generalize the classic Karhunen-Loève theorem in a very neat way. We first introduce the framework and basics notations from \cite{FM-2009}. 

Let $S$ be a nonempty set and $K$ a positive definite kernel on $S$, i.e., a function $ K: S \times S \longrightarrow \mathbb{C} $ satisfying the inequality 
$$
\displaystyle
	\sum_{i,j}^{n} \overline{c_{i}} c_{j} K(x_{i}, x_{j}) \geq 0,
$$
whenever $n \geq 1$, $\{ x_{1}, x_{2}, \ldots, x_{n} \}$ is a subset of $S$ and $\{ c_{1}, c_{2}, \ldots, c_{n} \}$ is a subset of $\mathbb{C}$. The set of all positive definite kernels over $S$ is denoted by $\mathcal{PD}(S)$. 

If $S$ is endowed with a measure $\upsilon$, denote by $L^{2}\mathcal{PD}(S, \upsilon)$ the class of kernels such that the associated integral operator 
$$ \displaystyle
	\mathcal{K}(f)(x) := \int_{S} K(x,y) f(y) d\upsilon(y), \ \ \ f \in L^{2}(S, \upsilon), \ \ x \in S,
$$
is positive, that is, when the following conditions holds
\begin{equation*}
\displaystyle
\int_{S} \left( \int_{S} K(x,y) f(y) d\upsilon(y) \right) \overline{f(x)} d\upsilon(x) \ \geq 0, \ \ \ f \in L^{2}(S, \upsilon),
\end{equation*}
	i.e.,
\begin{equation*}
	\left\langle \mathcal{K}(f), f \right\rangle_{L^{2}(S, \upsilon)}  \ \geq 0, \ \ \ f \in L^{2}(S, \upsilon).
\end{equation*}

	Finally we define what is a  Mercer's kernel according to \cite{FM-2013}: 
	A continuous kernel $K$ on $S$ is a Mercer's kernel when it possesses a series representation of the form 
\begin{equation*} \label{Mer_ker_rep_1}
\displaystyle
	K(x,y) = \sum_{j=1}^{\infty} a_{j}(\mathcal{K}) \zeta_{j}(x) \overline{\zeta_{j}(y)}, \ \ \ x, y \in S,
\end{equation*} 
where $\{\zeta_{j}\}_{j \in \mathbb{N}}$ is an $L^{2}(S, \upsilon)$-orthonormal basis of continuous functions on $S$, $\left\lbrace a_{j}(\mathcal{K}) \right\rbrace _{ j \in \mathbb{N} }$ decreases to $0$ and the series converges uniformly and absolutely on compact subsets of $S \times S$. 

For the rest of this manuscript we will consider $S$ to be a topological space endowed with a strictly positive measure $\upsilon$, that is, a complete Borel measure on $S$ for which two properties hold: every open nonempty subset of $S$ has positive measure and every $x \in S$ belongs to an open subset of $S$ having finite measure. Besides, as in Section \ref{Preli}, $(\Omega, \mathcal{F}, \mathbb{P})$ denotes a complete probability space.

\hspace{0.1cm}

\begin{thm} \label{GKL} 
Let $(X_{s})_{s \in S}$ be a complex-valued centred stochastic process with continuous covariance function and such that 
\begin{equation} \label{FM_cond-1}
\displaystyle
	\int_{S} \mathbb{E} (|X_{s}|^{2}) d\upsilon(s) \ < \ +\infty.
\end{equation}
Then, the kernel $K$ associated with the covariance function of $X$ is a Mercer's Kernel. Therefore, $X$ admits a Karhunen-Loève expansion 
\begin{equation} \label{FM_exp-1}
\displaystyle
	X = \sum_{j=1}^{+\infty} \lambda_{j} \zeta_{j}
\end{equation}
where $\{ \zeta_{j} \}_{j \in \mathbb{N}}$ is an orthonormal basis of $L^{2}(S, \upsilon)$, 
\begin{equation*} \label{FM_coeff-1}
\displaystyle
	\lambda_{j} = \int_{S} X_{s} \zeta_{j}(s) \ d\upsilon(s),
\end{equation*}
with $\mathbb{E}[\lambda_{j}] = 0$, and there exists a sequence $\{ a_{j} \}_{j \in \mathbb{N}} $ of non-negative real numbers such that $\mathbb{E}[\lambda_{j} \overline{ \lambda_{k} } ] = \delta_{jk} \, a_{j}$ and ${\rm Var} [\lambda_{j}] = a_{j}$.

The series expansion (\ref{FM_exp-1}) converges in $L^{2}(\Omega \times S; \mathbb{C})$, i.e., 
\begin{equation*}
	\displaystyle
	\lim_{J \rightarrow +\infty} \mathbb{E} \left( \int_{S} \left( X_{s} - \sum_{j=1}^{J} \lambda_{j} \zeta_{j}(s) \right)^{2} d\upsilon(s) \right) = 0. 
\end{equation*}

The series expansion (\ref{FM_exp-1}) converges in $L^{2}(\Omega; \mathbb{C})$ for all $s \in S$,  i.e., for all $s \in S$
\begin{equation*}
	\displaystyle
	\lim_{J \rightarrow +\infty} \mathbb{E} \left( \left( X_{s} - \sum_{j=1}^{J} \lambda_{j} \zeta_{j}(s) \right)^{2} \right) = 0. 
\end{equation*}

The convergence of the series expansion (\ref{FM_exp-1}) is absolute and uniform on compact subsets of $S$ in the mean-square sense.
\end{thm}

\begin{proof}

Let $K$ be the kernel associated to the covariance of the stochastic process $X$, i.e. 
\begin{eqnarray*} 
	K: S \times S &\longrightarrow & \mathbb{C} \\
		(t, s) &\longrightarrow & K(t, s) = \mathbb{E}(X_{t} \overline{ X_{s} } ), 	
\end{eqnarray*}
and let $\mathcal{K}$ be it's associated integral operator. From hypothesis (\ref{FM_cond-1}) it is direct to see that the mapping $\kappa$, such that $s \in S \mapsto  \mathbb{C} \ni \kappa(s) := K(s,s)$, belongs to $L^{1}(S, \upsilon)$. Since $K$ is positive definite in the usual sense, $K(s,t) = \overline{K(t,s)}$ and the matrix 
\[
\begin{bmatrix}
K(s,s) & K(s,t) \\
\overline{K(s,t)} & K(t,t)  
\end{bmatrix}
\]
is positive definite, hence its determinant $K(s,s) K(t,t) - K(s,t)\overline{K(s,t)}$ is non-negative. Thus
\begin{eqnarray*}
\displaystyle
	\int_{S \times S} |K(s,t)|^{2} d(\upsilon \otimes \upsilon) (s,t) &\leq & \int_{S \times S} K(s,s) K(t,t)  {\rm d} (\upsilon \otimes \upsilon) (s,t) \\
	&=& \left( \int_{S} \kappa(s) {\rm d} \upsilon (s) \right)^{2} \ < \ \infty,
\end{eqnarray*}
that is, $K \in L^{2}(S \times S, \upsilon \otimes \upsilon)$.

	Now, consider $f, g \in L^{2}(S, \upsilon)$ and the classic tensor product of functions, i.e., 
$$ 
(f \otimes g)(s,t) = f(s) g(t), \ \ \ s,t \in S,
$$
with inner product given by 
$$ 
\left\langle f_{1} \otimes f_{2} , g_{1} \otimes g_{2} \right\rangle _{L^{2}(S \times S, \upsilon \times \upsilon)} \ = \ \left\langle f_{1} , g_{1}\right\rangle _{L^{2}(S, \upsilon)} \left\langle f_{2} , g_{2}\right\rangle _{L^{2}(S, \upsilon)}.
$$

Apparently, $f \otimes g \in L^{2}(S \times S, \upsilon \otimes \upsilon)$, and thus, by Cauchy-Schwarz inequality we obtain $K \cdot f \otimes g \ \in L^{1}(S \times S,\upsilon \otimes \upsilon)$. This last condition allows to use Fubini's theorem, so for $f \in L^{2}(S, \upsilon)$ 
\begin{eqnarray*}
\displaystyle
	\left\langle\mathcal{K}(f), f\right\rangle_{2} &=& \int_{S} \left( \int_{S} K(s,t) f(t) \, {\rm d} \upsilon(t) \right) \overline{ f(s) } \, {\rm d}\upsilon(s) \\
	&=& \int_{S} \int_{S} K(s,t) f(t) \overline{ f(s) } \, {\rm d} \upsilon(t) {\rm d} \, \upsilon(s) \\
	&=& \int_{S} \int_{S} \mathbb{E}(X_{s} \overline{ X_{t} } ) f(t) \overline{ f(s) } \, {\rm d} \upsilon(t) \, {\rm d} \upsilon(s) \\
	&=& \mathbb{E} \int_{S} \int_{S} X_{s} \overline { f(s) } \ \overline{ X_{t} } f(t) \, {\rm d} \upsilon(t) \, {\rm d} \upsilon(s) \\
	&=& \mathbb{E} \left\langle X , f \right\rangle_{2} \left\langle f , X \right\rangle_{2} \\
	&=& \mathbb{E} \left\langle X , f \right\rangle_{2}  \overline { \left\langle X , f \right\rangle_{2} } \\
	&=& \mathbb{E} | \left\langle X , f \right\rangle_{2} | ^{2} \ \geq \ 0,
\end{eqnarray*}
hence, $K \in L^{2}\mathcal{P}\mathcal{D}(S, \upsilon)$.

In conclusion, the kernel $K$ is continuous and $L^{2}(S, \upsilon)$-positive definite on $S$, and the mapping $\kappa$ belongs to $L^{1}(S, \upsilon)$. Then, by theorem 3.1 in \cite{FM-2013} $K$ is a Mercer's kernel. 

	The rest of the proof concerns the Karhunen-Loève expansion of the process $X$ and it follows well-known arguments that we reproduce for the convenience of the reader. We have that $K$ has a $L^{2}(S \times S, \upsilon \otimes \upsilon)$-convergent series representation in the form
\begin{equation*} \label{FM_ker_rep-1}
\displaystyle 
	K(t,s) = \sum_{j=1} ^{+\infty} a_{j}(\mathcal{K}) \zeta_{j}(t) \overline{\zeta_{j}(s)}, \ \ \ t, s \ \in S,
\end{equation*} 
 $\left\lbrace a_{j}(\mathcal{K}) \right\rbrace _{j \in \mathbb{N}}$ decreases to $0$, $\{ \zeta_{j}\}_{j \in \mathbb{N}}$ is a $L^{2}(S, \upsilon)$-orthonormal basis. 
The convergence of the series is absolute  and uniform on compact subsets of $S \times S$. 

From condition (\ref{FM_cond-1}) there exists a set $\Omega_{0} \subseteq \Omega$ with $\mathbb{P}(\Omega_{0}) = 1$ such that for all $\omega \in \Omega_{0}$, the mapping $s \in S \mapsto  X_{s}(\omega)$ is in $L^{2}(S, \upsilon)$. Define the random coefficients
\begin{equation*} \label{FM_coeff-2}
\displaystyle
	\lambda_{j} = \int_{S} X_{t} \overline{ \zeta_{j}(t) } \, {\rm d} \upsilon(t) = \left\langle X, \zeta_{j} \right\rangle_{L^{2}(S, \upsilon)}.
\end{equation*}
Note that $\mathbb{E}(\lambda_{j}^{2}) = \left\langle\mathcal{K}(\zeta_{j}) , \zeta_{j} \right\rangle_{2} $, hence Cauchy-Scharwz inequality guarantees that $\lambda_{j} \in L^{2}(\Omega)$ for all $j$. Also,  Fubini's theorem allows to see that $\mathbb{E}[\lambda_{j}] = 0$, $\mathbb{E}[\lambda_{j} \overline{  \lambda_{k} } ] = \delta_{j k} a_{j}$ and ${\rm Var} [\lambda_{j}] = a_{j}$.

Now, for any fixed $\omega \in \Omega_{0}$ it is clear that
\begin{equation*} 
\displaystyle
	f_{J}(\omega) := \int_{S} \left( X_{s}(\omega) - \sum_{j=1}^{J} \lambda_{j}(\omega) \zeta_{j} \right)^{2} {\rm d} \upsilon(s) \ \xrightarrow[J \rightarrow + \infty]{} \ 0.
\end{equation*}
By orthogonality of the $\left( \zeta_{j} \right)$ we observe that
\begin{eqnarray*}
\displaystyle
	0 \leq f_{J}(\omega) &:=& \int_{S} \bigg| X_{s}(\omega) - \sum_{j=1}^{J} \lambda_{j}(\omega) \zeta_{j} \bigg| ^{2} {\rm d} \upsilon(s) \\ 
	&=& \int_{S} | X_{s}(\omega) | ^{2} {\rm d} \upsilon(s) - \sum_{j=1}^{J} | \lambda_{j}(\omega) | ^{2} \\
	&\leq &  \int_{S} | X_{s}(\omega) | ^{2} {\rm d} \upsilon(s) - \sum_{j=1}^{J-1} | \lambda_{j}(\omega) |  ^{2} := f_{J-1}(\omega), 	
\end{eqnarray*}
therefore, 
\begin{equation*}
\displaystyle
	|f_{J}(\omega)| \leq |f_{0}(\omega)| := \int_{S} |X_{s}(\omega) | ^{2} {\rm d} \upsilon(s), \ \ \ J \in \mathbb{N}.
\end{equation*}
By condition (\ref{FM_cond-1}) is clear that $\mathbb{E}(|f_{0}|) < \infty $, hence, the dominated convergence theorem allows us to conclude that 
\begin{equation*}
\displaystyle
	\mathbb{E}(|f_{J}|) = \mathbb{E} \left( \int_{S} \bigg| X_{s}(\omega) - \sum_{j=1}^{J} \lambda_{j}(\omega) \zeta_{j} \bigg| ^{2} d\upsilon(s) \right) \xrightarrow[J \rightarrow + \infty]{} \ 0.
\end{equation*}
Now, fix $s \in S$. Again by Fubini's theorem we observe that,
\begin{eqnarray*}
\displaystyle
& & \mathbb{E} \bigg| X_{t} - \sum_{j=1} ^{n} \lambda_{j} \zeta_{j} (t) \bigg| ^{2} \\
&=& \mathbb{E} | X_{t} | ^{2} 
	-\sum_{j=1} ^{n} \overline{ \zeta_{j}(t) } \mathbb{E} \left( X_{t} \overline{ \lambda_{j} } \right)  
	- \sum_{j=1} ^{n} \zeta_{j}(t) \mathbb{E} \left( \overline{ X_{t} } \lambda_{j} \right) 
	+ \sum_{j, k = 1} ^{n}  \zeta_{j}(t) \overline { \zeta_{k} (t) } \mathbb{E} \left( \lambda_{j} \overline{ \lambda_{k} }   \right) \\
&=& \mathbb{E} | X_{t} | ^{2} 
	- \sum_{j=1} ^{n} \overline{ \zeta_{j}(t) } \mathbb{E} \left( X_{t} \int_{S} \overline{ X_{s} } \zeta_{j}(s) \, {\rm d} \upsilon(s) \right)  \\
& & - \sum_{j=1} ^{n} \zeta_{j}(t) \mathbb{E} \left( \overline{ X_{t} } \int_{S} X_{s} \overline{ \zeta_{j}(s) } \, {\rm d} \upsilon(s) \right) 
	+ \sum_{j, k = 1} ^{n}  \zeta_{j}(t) \overline{ \zeta_{k} (t) } \delta_{kj} a_{k}  \\
&=& \mathbb{E} | X_{t} | ^{2} 
	- \sum_{j=1} ^{n} \overline{ \zeta_{j}(t) } \int_{S} \mathbb{E} \left( X_{t} \overline{ X_{s} } \right) \zeta_{j}(s) \, {\rm d} \upsilon(s)  \\
& &	- \sum_{j=1} ^{n} \zeta_{j}(t) \int_{S} \mathbb{E} \left( \overline{ X_{t} } X_{s} \right) \overline{ \zeta_{j}(s) } \, {\rm d} \upsilon(s)	
	 + \sum_{ j = 1} ^{n} | \zeta_{j} (t) | ^{2}  a_{j}  \\
&=& K(t,t) - \sum_{j=1} ^{n} \overline{ \zeta_{j}(t) } \int_{S} K(t,s) \zeta_{j}(s) \, {\rm d} \upsilon(s)  \\
& &	- \sum_{j=1} ^{n} \zeta_{j}(t) \int_{S} K(s,t) \overline { \zeta_{j}(s) } \, {\rm d} \upsilon(s) 
	+ \sum_{ j = 1} ^{n} | \zeta_{j} (t) | ^{2}  a_{j}  \\
&=& K(t,t) - \sum_{j=1} ^{n} \overline{ \zeta_{j}(t) } \int_{S} K(t,s) \zeta_{j}(s) \, {\rm d} \upsilon(s)   \\
& &	- \sum_{j=1} ^{n} \zeta_{j}(t) \int_{S} \overline { K(t,s) } \overline { \zeta_{j}(s) } \, {\rm d} \upsilon(s) 
	+ \sum_{ j = 1} ^{n} | \zeta_{j} (t) | ^{2}  a_{j}  \\	
&=& K(t,t) - \sum_{j=1} ^{n} \overline{ \zeta_{j}(t) } \mathcal{K} \zeta_{j}(t)
	- \sum_{j=1} ^{n} \zeta_{j}(t) \overline { \mathcal{K} \zeta_{j}(t) }
	+ \sum_{ j= 1} ^{n} | \zeta_{j} (t) | ^{2}  a_{j}  \\
&=& K(t,t) - 2\sum_{j=1} ^{n} | \zeta_{j} (t) |^{2} a_{j} + \sum_{ j = 1} ^{n} | \zeta_{j} (t) | ^{2}  a_{j} \\
&=& K(t,t) - \sum_{j=1} ^{n} | \zeta_{j} (t) |^{2} a_{j} \ \
 \xrightarrow[n \rightarrow + \infty] \ 0. 
\end{eqnarray*} 
Thus, the proof is concluded. 
\end{proof}

\begin{rem} Karhunen-Loève expansion (or Karhunen-Loève theorem), usually require extra hypothesis, like compactness of the associated space $S$ or some kind of invariance of the field. In that line the Stochastic Peter-Weyl theorem (theorem 5.5 introduced in \cite{MP-2011}) may be understood as a Karhunen-Loève theorem for 2-weakly isotropic fields over $G$ a topological compact group with associated Haar measure of unit mass. 
Theorem \ref{GKL} only require condition (\ref{FM_cond-1}) and the continuity of the covariance function.



\end{rem}


\section*{Acknowledgments}

We are indebted to the Editor and to three Referees, whose thorough reviews allowed for a considerably improved version of the manuscript. 



\end{document}